\documentclass[11pt,reqno,british]{amsart}
\usepackage[T1]{fontenc}
\usepackage[latin9]{inputenc}
\usepackage[a4paper]{geometry}
\usepackage{xcolor}
\usepackage{pdfcolmk}
\usepackage{mathtools}
\usepackage{amstext}
\usepackage{amsthm}
\usepackage{amssymb}
\usepackage{graphicx}
\usepackage{setspace}
\PassOptionsToPackage{normalem}{ulem}
\usepackage{ulem}
\makeatletter

\providecolor{lyxadded}{rgb}{0,0,1}
\providecolor{lyxdeleted}{rgb}{1,0,0}

\DeclareRobustCommand{\lyxsout}[1]{\ifx\\#1\else\sout{#1}\fi}

\numberwithin{equation}{section}
\numberwithin{figure}{section}
\theoremstyle{plain}
\newtheorem{thm}{\protect\theoremname}[section]
\theoremstyle{definition}
\newtheorem{example}[thm]{\protect\examplename}
\theoremstyle{plain}
\newtheorem{cor}[thm]{\protect\corollaryname}
\theoremstyle{plain}
\newtheorem{lem}[thm]{\protect\lemmaname}
\theoremstyle{plain}
\newtheorem*{fact*}{\protect\factname}
\theoremstyle{remark}
\newtheorem{rem}[thm]{\protect\remarkname}
\theoremstyle{plain}
\newtheorem{prop}[thm]{\protect\propositionname}

\usepackage{amscd}
\usepackage{mathptmx}
\usepackage{bbm}
\usepackage{ifthen}

\newcommand{\e}{\mathrm{e}}
\newcommand{\1}{\mathbbm{1}}
\newcommand{\N}{\mathbb{N}}

\newcommand{\Z}{\mathbb{Z}}

\newcommand{\R}{\mathbb{R}}

\renewcommand{\Pi}{\pi}
\renewcommand{\emptyset}{\varnothing}

\DeclareMathOperator*{\card}{card}

\usepackage{mathtools}
\usepackage{pgf,tikz}
\usepackage{pgfplots}
\usepackage{mathrsfs}
\usetikzlibrary{arrows}

\newcommand{\emptyword}{\emptyset}
\renewcommand{\emptyset}{\varnothing}
\renewcommand{\tilde}{\widetilde}
\newcommand{\Rep}{\mathcal{R}}
\newcommand{\FI}{F} 
\newcommand{\D}{D_{\psi}} 

\usepackage{hyperref}
\@ifundefined{textmu}
 {\usepackage{textcomp}}{}

\makeatother

\usepackage{babel}
\providecommand{\corollaryname}{Corollary}
\providecommand{\examplename}{Example}
\providecommand{\factname}{Fact}
\providecommand{\lemmaname}{Lemma}
\providecommand{\propositionname}{Proposition}
\providecommand{\remarkname}{Remark}
\providecommand{\theoremname}{Theorem}

\begin{document}
\title{Thermodynamic formalism for transient dynamics on the real line}
\author{Maik Gröger}
\address{Faculty of Mathematics, University of Vienna, Oskar Morgensternplatz
1, 1090 Vienna, Austria}
\email{maik.groeger@univie.ac.at}
\author{Johannes Jaerisch}
\address{Graduate School of Mathematics, Nagoya University, Furocho, Chikusaku,
Nagoya, 464-8602 Japan}
\email{jaerisch@math.nagoya-u.ac.jp}
\author{Marc Kesseböhmer}
\address{FB03 -- Mathematik und Informatik, Universität Bremen, 28359 Bremen,
Germany}
\email{mhk@math.uni-bremen.de}
\begin{abstract}
We develop a new thermodynamic formalism to investigate the transient
behaviour of maps on the real line which are skew-periodic $\Z$-extensions
of expanding interval maps. Our main focus lies in the dimensional
analysis of the recurrent and transient sets as well as in determining
the whole dimension spectrum with respect to $\alpha$-escaping sets.
Our results provide a one-dimensional model for the phenomenon of
a dimension gap occurring for limit sets of Kleinian groups. In particular,
we show that a dimension gap occurs if and only if we have non-zero
drift and we are able to precisely quantify its width as an application
of our new formalism.
\end{abstract}

\thanks{MG was supported by the DFG grants JA 1721/2-1 and GR 4899/1-1, JJ
was supported by JSPS KAKENHI 17K14203 and MK acknowledges support
by the DFG grant KE 1440/3-1. This project was also part of the activities
of the Scientific Network \textquotedblleft Skew product dynamics
and multifractal analysis\textquotedblright{} (DFG grant OE 538/3-1).}
\maketitle

\section{Introduction}

The main motivation of this article is the connection between transient
phenomena of dynamical systems and its manifestation in dimensional
quantities. Since transience can impose major obstructions to an ergodic-theoretic
description of (fractal-)geometric features its further understanding
is vital and has attracted a lot of attention. For instance, it has
a strong tradition in complex dynamics with landmark results like
the ones obtained for complex quadratic polynomials in \cite{MR1626737}
or \cite{MR2373353}. A closely related and paralleling line of research
established corresponding results for families of transcendental functions,
see for example \cite{MR2302520,MR2465667,MR2197375,MR871679}. In
both cases, a particular striking effect revealing the preponderance
of transience is the occurrence of a so-called dimension gap. In fact,
the origin of this phenomenon goes back to the rich field of geometric
group theory which we explain in more detail further below.

In the framework of thermodynamic formalism, transient effects in
topological Markov chains have been seminally studied by Sarig \cite{MR1738951,MR1818392}.
Directly related to this are fractal-geometric applications of thermodynamic
formalism for infinite conformal graph directed Markov systems which
have been systematically worked out by Mauldin and Urbanski in \cite{MR2003772}.
In there, strong mixing conditions were introduced to guarantee that
the recurrent behaviour governs the system. One main goal of this
paper is to set up a new thermodynamic formalism in the absence of
such strong mixing conditions to provide a systematic approach to
the geometric phenomenon of a dimension gap. More precisely, we introduce
the concept of fibre-induced pressure which allows us to express the
occurrence and the width of a dimension gap for skew-periodic $\Z$-extensions
of expanding interval maps exclusively in terms of this newly developed
pressure. Furthermore, we obtain effective analytic relations between
the fibre-induced pressure and the classical pressure of the base
transformation allowing us to determine the crucial dimensional quantities
in a number of examples explicitly.

Let us now illustrate the phenomenon of a dimension gap in the setting
of actions of non-elementary Kleinian groups $G$ on the hyperbolic
space $\mathbb{H}^{n}$. By a general result of Bishop and Jones \cite{MR1484767}
we know that the Hausdorff dimension of both the radial limit set
$\Lambda_{r}\left(G\right)$ and the uniformly radial limit set $\Lambda_{ur}\left(G\right)$
of $G$ are equal to the Poincaré exponent of $G$ given by
\begin{equation}
\delta_{G}\coloneqq\inf\left\{ s\geq0:\sum_{g\in G}\e^{-s\cdot d_{H}\left(0,g0\right)}<\infty\right\} ,\label{eq:definition Poincare exponent group}
\end{equation}
where $d_{H}$ denotes the hyperbolic distance on $\mathbb{H}^{n}$.
Recall that $\Lambda_{r}\left(G\right)$ and $\Lambda_{ur}\left(G\right)$
represent recurrent dynamics of the geodesic flow on $\mathbb{H}^{n}/G$.
Clearly, for a normal subgroup $N<G$ we have that $\delta_{G}\geq\delta_{N}$
and moreover, 
\[
\dim_{H}\left(\Lambda_{r}\left(G\right)\right)=\delta_{G}>\delta_{N}=\dim_{H}\left(\Lambda_{r}\left(N\right)\right)\iff G/N\;\text{is non-amenable,}
\]
with $\dim_{H}(\,\cdot\,)$ the Hausdorff dimension of the corresponding
set. This was first proved by Brooks for certain Kleinian groups fulfilling
$\delta_{G}>(n-1)/2$ in \cite{MR783536} and later generalised to
a wider class of groups without this restriction by Stadlbauer \cite{Stadlbauer11}.
Note that by a result of Falk and Stratmann $\delta_{N}\ge\delta_{G}/2$,
see \cite{MR2097162}. If $G$ is additionally geometrically finite,
then the strict inequality $\delta_{N}>\delta_{G}/2$ holds by a result
of Roblin \cite{MR2166367} (see also \cite{MR3299281}). Furthermore,
if $\Lambda\left(G\right)$ denotes the limit set of the Kleinian
group $G$, then $\delta_{G}=\dim_{H}\left(\Lambda\left(G\right)\right)$
and, since $\Lambda\left(N\right)=\Lambda\left(G\right)$, this implies
the following criterion for the occurrence of a \emph{dimension gap}:
\[
\dim_{H}\left(\Lambda_{r}\left(N\right)\right)=\dim_{H}\left(\Lambda_{ur}\left(N\right)\right)<\dim_{H}\left(\Lambda\left(N\right)\right)\iff G/N\;\text{is non-amenable.}
\]
In other words, a certain amount of transient behaviour causes a dimension
gap from the dimension of the full limit set compared to the restriction
of the limit set to certain recurrent parts. It is remarkable that
for Kleinian groups the presence of a dimension gap depends only on
the group-theoretic property of\emph{ }amenability. Accordingly, a
natural example for the occurrence of a dimension gap is given by
a Schottky group $G=N\rtimes\mathbb{F}_{2}$ where $\mathbb{F}_{2}$
denotes the free group with two generators. Nevertheless, only little
is known in the literature concerning the concrete size of this dimension
gap.

The occurrence of a dimension gap is closely related to the decay
of certain return probabilities. In fact, Kesten \cite{MR0112053,MR0109367}
has shown for symmetric random walks on countable groups that exponential
decay of return probabilities is equivalent to non-amenability. However,
for amenable groups exponential decay can also be caused by non-symmetric
random walks. To be more precise, for groups admitting a recurrent
random walk (e.g.~$\Z$) it is shown in \cite{MR3436756} that exponential
decay of return probabilities is equivalent to a lack of certain symmetry
condition on the thermodynamic potential related to the random walk
(see also Remark \ref{rem:characterisation of recurrence} for further
details).

We are aiming at investigating these closely linked phenomena for
a class of maps on $\R$ which can be considered as models of $\Z$-extensions
of Kleinian groups. In fact, if the Kleinian group $G=N\rtimes\Z$
is a Schottky group, then the elements in $\Lambda_{r}\left(N\right)$
can be characterised as the limits of $G$-orbits for which the $\Z$-coordinate
returns infinitely often to some point in $\mathbb{Z}$ (compare this
with our definition of a recurrent set, see Section \ref{subsec:Recurrent-and-transient}).
Since $\Z$ is amenable, these limit points have full Hausdorff dimension
by Brooks' amenability criterion. We will see later (end of Section
\ref{sec:Examples}) that this property also follows from the fact
that the $\Z$-coordinate has zero drift with respect to a canonical
invariant measure obtained from the Patterson-Sullivan construction.

Our models witness drift behaviour and we show that indeed non-zero
drift is equivalent to the occurrence of a dimension gap, see Theorem
\ref{thm:-dimension gap}. It is therefore also very natural to consider
subsets of the transient dynamics with fixed drift in more detail.
This motivates the definitions of various escaping sets in our one-dimensional
models, see Section \ref{subsec:Escaping-sets}. The related dimension
spectra will allow us to determine the size of the dimension gap explicitly.
Similar results will be shown in the forthcoming paper \cite{JKG20}
on $\Z$-extensions with reflective boundaries allowing us to illuminate
earlier results in \cite{MR1438267,MR2959300,MR3610938} which studied
a family suggested by van Strien modelling induced maps of Fibonacci
unimodal maps. Let us point out that drift arguments where also prominent
in the proofs of \cite{MR2959300}.

Our leading motivating example for which we obtain dimensional results
on the transient behaviour stems from the family of (a-)symmetric
random walks, see Example \ref{exa:Classical Random Walk}. More precisely,
let $\FI$ be an expanding interval map with finitely many $C^{1+\epsilon}$
full branches, say $\left.\FI\right|_{I_{i}}:I_{i}\rightarrow[0,1]$
with disjoint intervals $I_{i}\subset[0,1]$ with non-empty interior,
$i\in I\coloneqq\{1,\dots,m\}$, $m\geq2$. Set $h_{i}\coloneqq H_{i}^{-1}:\left[0,1\right]\to\overline{I_{i}}$
for the continuous continuation $H_{i}$ of $\left.\FI\right|_{I_{i}}$
to $\overline{I_{i}}$ and define the corresponding coding map $\pi:\Sigma\coloneqq I^{\N}\rightarrow[0,1]$
by $\pi\left(\omega_{1},\omega_{2},\ldots\right)\coloneqq x$ for
$\bigcap_{n\in\N}h_{\omega_{1}}\circ\cdots\circ h_{\omega_{n}}\left(\left[0,1\right]\right)=\left\{ x\right\} $.
The \emph{repeller} of $F$ is then $\pi\left(\Sigma\right)\subset[0,1]$
and we set $\Rep\coloneqq\pi\left(\Sigma\right)\setminus\mathcal{D}+\Z$,
where we subtract the countable set $\mathcal{D}$ to avoid technical
problems stemming from possible discontinuities at the boundaries
of the $I_{i}$'s, see (\ref{eq:D}) for the precise definition of
$\mathcal{D}$. We assume that the \emph{step length function} $\Psi:[0,1]\rightarrow\Z$
is constant on each of the intervals $I_{i}$ and consider the \emph{$\Psi$-lift}
of $\FI$ given by 
\begin{align*}
\FI_{\,\Psi}:\Rep & \rightarrow\Rep:\quad x\mapsto\sum_{k\in\Z}\left(k+F(x-k)+\Psi(x-k)\right)\1_{[0,1]}(x-k).
\end{align*}
Since this function is periodic up to a certain skewness induced by
$\Psi$, we refer to this map as a \emph{skew-periodic interval map}.
In order to study the various escaping sets of $\FI_{\,\Psi}$ later
it will be necessary to consider $\R$-extensions rather than $\Z$-extensions.
This is the reason why in our abstract set-up we will consider \emph{skew
product dynamical systems} defined via
\begin{equation}
\sigma\rtimes f:\Sigma\times\R\rightarrow\Sigma\times\R,\quad(\sigma\rtimes f)(\omega,x)\coloneqq(\sigma(\omega),x+f(\omega)),\label{eq:definition skew product dynamical system}
\end{equation}
for some function $f:\Sigma\to\R$. We want to stress that after neglecting
a countable set $\FI_{\,\Psi}$ is topological conjugate to the $\Z$-extension
$\sigma\rtimes\left(\Psi\circ\pi\right):\Sigma\times\Z\to\Sigma\times\Z$,
see Lemma \ref{fac:factor}.
\begin{example}[One-dimensional random walk]
\label{exa:Classical Random Walk} We model a classical one-step
random walk via a skew-periodic interval map $\FI_{\,\Psi}$. For
this fix $c_{1},c_{2}\in\left(0,1\right)$ with $c_{1}+c_{2}\leq1$
and consider the map 
\[
F:x\mapsto\begin{cases}
c_{1}^{-1}x & \text{for }x\in[0,c_{1}]\\
c_{2}^{-1}x & \text{for }x\in\left(1-c_{2},1\right]
\end{cases},
\]
with code space $\Sigma:=\left\{ 1,2\right\} ^{\N}$ and set $\Psi\coloneqq-\1_{[0,c_{1}]}+\1_{\left(1-c_{2},1\right]}$
(see Figure \ref{fig:Classical-Random-Walk}). In this setting we
will also refer to $F_{\,\Psi}:\Rep\to\Rep$ as the\emph{ random walk
model}. For any $\left(p_{1},p_{2}\right)$-Bernoulli measure $\mu$
on $\Sigma$, we can model the classical (a-)symmetric random walk
on $\Z$ with transition probability $p_{1}$ to go one step left
and probability $p_{2}$ one step right. The random walk starting
in $0\in\mathbb{\Z}$ would then be given by the stochastic process
$\left(\left\lfloor F_{\,\Psi}^{n}\right\rfloor \right)_{n\in\N}$
with respect to the probability measure $\mu\circ\pi^{-1}$ on the
repeller $\pi\left(\Sigma\right)$. Note that $\pi\left(\Sigma\right)$
is a Cantor set with Hausdorff dimension $\delta<1$ if and only if
$c_{1}+c_{2}<1$ where $\delta$ is the unique number $s$ with $c_{1}^{s}+c_{2}^{s}=1$.
Otherwise, $\pi\left(\Sigma\right)$ is the unit interval and hence,
$\Rep=\R\setminus(\mathcal{D}+\Z)$.
\begin{figure}[h]
\begin{tikzpicture}[line cap=round,line join=round,>=triangle 45,x=.85cm,y=0.85cm] \draw[->,color=black] (-3.9,0.) -- (3.9,0.); \foreach \x in {-3,-2,-1,1,2,3} \draw[shift={(\x,0)},color=black] (0pt,2pt) -- (0pt,-2pt) node[below] {\scriptsize $\x$}; \draw[->,color=black] (0.,-3.9) -- (0.,3.9); \foreach \y in {-3,-2,-1,1,2,3} \draw[shift={(0,\y)},color=black] (2pt,0pt) -- (-2pt,0pt) node[left] {\scriptsize $\y$}; \draw[color=black] (0pt,-10pt) node[right] {\scriptsize $0$}; \clip(-3.9,-3.9) rectangle (3.9,3.9); 

\draw[line width=1.2pt, samples=10,domain=-3:-2.6] 
plot(\x,{2.50*(\x)+3.5}); 
\draw[line width=1.2pt, samples=10,domain=-2.6:-2] 
plot(\x,{5/3*(\x)+7/3});
\draw[line width=1.2pt, samples=10,domain=-2:-1.6] 
plot(\x,{2.5*(\x)+2}); 
\draw[line width=1.2pt, samples=10,domain=-1.6:-1] 
plot(\x,{5/3*(\x)+5/3});

\draw[line width=1.2pt, samples=10,domain=-1:-0.6] 
plot(\x,{2.50*(\x)+0.5}); 
\draw[line width=1.2pt, samples=10,domain=-0.6:0] 
plot(\x,{5/3*(\x)+3/3});
\draw[line width=1.2pt, samples=10,domain=0:0.4] 
plot(\x,{2.5*(\x)-1}); 
\draw[line width=1.2pt, samples=10,domain=0.4:1] 
plot(\x,{5/3*(\x)+1/3});
\draw[line width=1.2pt, samples=10,domain=1:1.4] 
plot(\x,{2.50*(\x)-2.5}); 
\draw[line width=1.2pt, samples=10,domain=1.4: 2] 
plot(\x,{5/3*(\x)-1/3});
\draw[line width=1.2pt, samples=10,domain= 2:2.4] 
plot(\x,{2.5*(\x)-4}); 
\draw[line width=1.2pt, samples=10,domain=2.4:3] 
plot(\x,{5/3*(\x)-3/3});
  \end{tikzpicture}

\caption{The graph of $F_{\,\Psi}$ modelling the (a-)symmetric random walk
with parameters $c_{1}=0.4$ and $c_{2}=0.6$.\label{fig:Classical-Random-Walk}}
\end{figure}
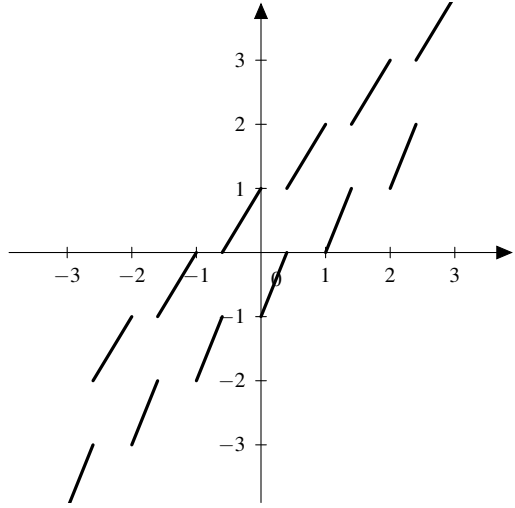

Crucial for our analysis will be the \emph{fibre-induced pressure
}$\mathcal{P}$ and its close connection to the so-called \emph{$\alpha$-Poincaré
exponent} $\delta_{\alpha}$ defined in the next section. Indeed,
the new pressure $\mathcal{P}$ is a natural generalisation of the
notion of Gurevich pressure (cf.~\cite{MR1738951} and Remark \ref{rem:original-gurevich})
and is necessary to perform our analysis for general escaping rates,
see Section \ref{subsec:Escaping-sets}. In particular, Gurevich pressure
is defined for topological Markov chains whereas our new notion is
defined for more general $\R$-extensions. Further, we will see that
our new quantity can also be deduced from the\emph{ classical pressure}
$\mathfrak{P}$. In fact, we will show in Theorem \ref{thm:fibre-induced pressure via base pressure}
below that 
\[
\mathcal{P}\left(f,\psi\right)=\inf_{s\in\R}\mathfrak{P}\left(s\psi+f\right),
\]
for $f,\psi:\Sigma\rightarrow\R$ Hölder continuous and $\psi$ satisfying
natural conditions fulfilled in our setting. Generalising the concept
of the classical Poincaré exponent, the $\alpha$-Poincaré exponent
will reveal a natural connection to the analysis of limit sets of
Kleinian groups as well as the dimension theory of Birkhoff averages,
see the remark after Theorem \ref{thm: Multifractal Decomposition}
and Remark \ref{rem:delta_=00005Calpha as Birkhoff average}.
\end{example}

\subsection{Main results}

We define the \emph{geometric potential} $\varphi:\Sigma\to(-\infty,0)$
in a Hölder continuous way such that $\varphi\left(\omega\right)\coloneqq-\log\left|\FI'\left(\pi\left(\omega\right)\right)\right|$
except possibly on a finite set. Denote by $\delta\text{\ensuremath{>0}}$
the unique $s$ such that $\mathfrak{P}\left(s\varphi\right)=0$.
Here, $\mathfrak{P}\left(f\right)$ refers to the classical topological
pressure defined for any continuous function $f:\Sigma\to\R,$ see
end of Section \ref{sec:Preliminaries-and-basic}. Then by \emph{Bowen's
formula} we have for the repeller $\pi(\Sigma)$ of $\FI$ that 
\begin{equation}
\delta=\dim_{H}(\pi(\Sigma)).\label{eq:Hausdorff dim Bowen formula}
\end{equation}
Further, we can associate to $\Psi$ a \emph{symbolic step length
function} $\psi:\Sigma\to\Z$ which is constant on one-cylinder sets
such that $\psi=\Psi\circ\pi$ except possibly on a finite set. Let
us introduce the\emph{ $\alpha$-Poincaré exponent}
\begin{equation}
\delta_{\alpha}:=\inf\left\{ s\ge0\mid\sum_{{\omega\in I^{\star}},\,{\left|S_{\omega}\left(\psi-\alpha\right)\right|\leq K}}\e^{s\cdot S_{\omega}\varphi}<\infty\right\} ,\label{eq:definition Poincare exponent}
\end{equation}
where $I^{\star}$ denotes the set of all finite words over the \emph{alphabet}
$I$ and $K>0$ is sufficiently large (the term $S_{\omega}f$, $\omega\in I^{\star}$
is defined in (\ref{eq:birkhoff sum and max sum}) for any $f:\Sigma\to\mathbb{R}$).
For this exponent we observe
\begin{equation}
\alpha\in\R\setminus\left[\underline{\psi},\overline{\psi}\right]\implies\delta_{\alpha}=0\label{eq:delta_alpha=00003D0}
\end{equation}
and for $\alpha\in\text{\ensuremath{\big[}}\underline{\psi},\overline{\psi}\big]$
the critical exponent $\delta_{\alpha}$ can be expressed as the unique
zero of the fibre-induced pressure for suitable potentials and is
positive on $(\underline{\psi},\overline{\psi})$ (see Theorem \ref{thm:critical exponent via base pressures}
as well as (\ref{eq:Psiupperlower}) for the definition of $\underline{\psi}$
and $\overline{\psi}$). We will also see that the map $\alpha\mapsto\delta_{\alpha}$
is real-analytic on $\big(\underline{\psi},\overline{\psi}\big)$
and unimodal but not necessarily concave (cf.~Section \ref{sec:Multifractal-decomposition}
and Section \ref{subsec:First-examples} for examples).

\subsubsection{Recurrent and transient sets and dimension gap\label{subsec:Recurrent-and-transient}}

For $\FI_{\,\Psi}$ we define the \emph{recurrent set} 
\[
\mathbf{R}:=\left\{ x\in\Rep\mid\exists K\in\R\:\text{such that }\left|\FI_{\,\Psi}^{n}(x)\right|\leq K\mbox{ for infinitely many }n\in\N\right\} 
\]
and the\emph{ uniform recurrent set }
\[
\mathbf{R}_{u}:=\left\{ x\in\Rep\mid\exists K\in\R\;\forall n\in\N\;\left|\FI_{\,\Psi}^{n}(x)\right|\leq K\right\} .
\]

The \emph{positive (resp. negative) transient set} is given by
\[
\mathbf{T}_{1}^{\pm}\coloneqq\left\{ x\in\Rep\mid\lim_{n\to\infty}\FI_{\,\Psi}^{n}(x)=\pm\infty\right\} .
\]
We also consider the supersets
\[
\mathbf{T}_{2}^{\pm}\coloneqq\left\{ x\in\Rep\mid\limsup_{n\to\infty}\mp\FI_{\,\Psi}^{n}(x)<\infty\right\} ,\,\,\ \mathbf{T}_{3}^{\pm}\left(r\right)\coloneqq\left\{ x\in\Rep\mid\forall n\geq0\;\pm\FI_{\,\Psi}^{n}(x)>r\right\} ,
\]
for some $r\in\R$. Observe that
\begin{equation}
\mathbf{T}_{1}^{\pm}\subseteq\mathbf{T}_{2}^{\pm}=\bigcup_{k\in\Z}\mathbf{T}_{3}^{\pm}\left(k\right).\label{eq:T2UnionT3}
\end{equation}
The following theorems will be a consequence of our general multifractal
decomposition for escaping sets presented in the next subsection.
The corresponding proofs can be found in Section \ref{sec:Multifractal-decomposition}.
The first theorem is the analogue of the result of Bishop and Jones
in our setting. Last, recall that for the Hölder continuous function
$\delta\varphi:\Sigma\to\R$ there exists a unique Gibbs measure $\mu_{\delta\varphi}$,
see end of Section \ref{sec:Preliminaries-and-basic}.
\begin{thm}
\label{thm:Bishop and Jones}Let $\FI_{\,\Psi}:\Rep\rightarrow\Rep$
be the $\Psi$-lift of an expanding interval map $\FI$. Then we have
for the recurrent and uniformly recurrent sets 
\[
\dim_{H}(\mathbf{R})=\dim_{H}(\mathbf{R}_{u})=\delta_{0}.
\]
\end{thm}

We say the system $\FI_{\,\Psi}$ has a \emph{dimension gap} if the
Hausdorff dimension $\delta_{0}$ of the (uniformly) recurrent set
is strictly less than the Hausdorff dimension $\delta$ of $\Rep$.
In fact, this is the case if and only if the system has a \emph{drift}
$\mu_{\delta\varphi}\left(\psi\right)\neq0$, see Theorem \ref{thm:-dimension gap}
below. Furthermore, we are able to provide direct methods to determine
$\delta_{0}$, see Theorem \ref{thm:critical exponent via base pressures}.
This allows us to easily calculate $\delta_{0}$ for our examples
and to precisely quantify the dimension gap, see for instance Figure
\ref{fig:delta_0Graph}.
\begin{thm}
\label{thm:-Transient-sets}For the transient sets the following implications
hold:
\begin{itemize}
\item $\mu_{\delta\varphi}\left(\psi\right)\geq0$ implies $\dim_{H}\left(\mathbf{T}^{-}\right)=\delta_{0}$
and $\dim_{H}\left(\mathbf{T}^{+}\right)=\delta$,
\item $\mu_{\delta\varphi}\left(\psi\right)\leq0$ implies $\dim_{H}\left(\mathbf{T}^{+}\right)=\delta_{0}$
and $\dim_{H}\left(\mathbf{T}^{-}\right)=\delta$,
\end{itemize}
where $\mathbf{T}^{\pm}$ may be chosen to be one of the sets $\mathbf{T}_{1}^{\pm}$,
$\mathbf{T}_{2}^{\pm}$ or $\mathbf{T}_{3}^{\pm}\left(r\right)$ for
any $r\in\R$.
\end{thm}

\subsubsection{\label{subsec:Escaping-sets}Escaping sets}

For $\alpha\in\R$ let us now define the $\alpha$\emph{-escaping
set} for $\FI_{\,\Psi}$ by 
\begin{align*}
\mathbf{E}(\alpha) & \coloneqq\left\{ x\in\Rep\mid\exists K>0\;\left|\FI_{\,\Psi}^{n}(x)-n\alpha\right|\leq K\;\text{for infinitely many }n\in\N\right\} ,
\end{align*}
and the \emph{uniformly }$\alpha$\emph{-escaping set} for $\FI_{\,\Psi}$
by
\begin{align*}
\mathbf{E}_{u}(\alpha): & =\left\{ x\in\Rep\mid\exists K>0\;\forall n\in\N\;\left|\FI_{\,\Psi}^{n}(x)-n\alpha\right|\leq K\right\} .
\end{align*}

As stated above, the following result proves our statement on the
occurrence of a dimension gap.
\begin{thm}
\label{thm:-dimension gap} We have $\delta_{\alpha}=\delta$ if and
only if $\mu_{\delta\varphi}(\psi)=\alpha$. In particular, a dimension
gap occurs if and only if $\mu_{\delta\varphi}(\psi)\neq0$.
\end{thm}

\begin{thm}
[Multifractal decomposition with respect to $\alpha$-escaping] \label{thm: Multifractal Decomposition}
For $\alpha\in\R$ we have 
\[
\dim_{H}(\mathbf{E}(\alpha))=\dim_{H}(\mathbf{E}_{u}(\alpha))=\delta_{\alpha}.
\]
\end{thm}

We note that $\delta_{\alpha}$ also allows for a multifractal spectral
interpretation of certain Birkhoff averages (cf.\ Remark \ref{rem:delta_=00005Calpha as Birkhoff average}).
However, the results of Theorem \ref{thm: Multifractal Decomposition}
need more subtle ideas adopted from the analysis of Kleinian groups
and we believe that this connection is also of independent interest.

\subsubsection{First examples and further consequences\label{subsec:First-examples}}

The previous theorems applied to Example \ref{exa:Classical Random Walk}
with fixed $c_{1,}c_{2}\in(0,1)$ lead to the following observations.
We determine the graph of $\alpha\mapsto\dim_{H}\left(\mathbf{E}\left(\alpha\right)\right)=\dim_{H}(\mathbf{E}_{u}(\alpha))=\delta_{\alpha}(c_{1},c_{2})$
for the random walk model, see Figure \ref{fig:asymmetric-Random-Walk-Spectrum}.
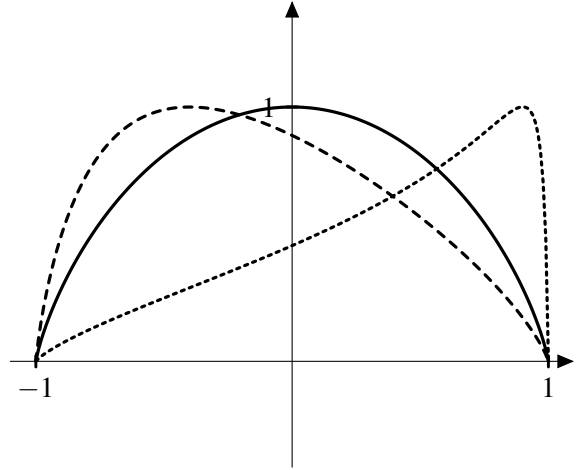
\begin{figure}[h]
\pgfplotsset{width=9cm,compat=1.9}
\begin{tikzpicture}[line cap=round,line join=round,>=triangle 45]
\begin{axis}[xmin=-1.1,xmax=1.1,     axis equal,     axis lines=middle,     axis line style={->},     tick style={color=black,line width=1.1pt},     xtick={-1,0,1}, xticklabels={$-1$,$0$,$1$}, ytick={0,1},yticklabels={$0$, $1$} ]     
\addplot 
[         domain=-1:1, samples=197,line width=1.2pt,  
] {(-ln(4.0)+(1.0+(x))*ln(1.0+(x))+(1.0-(x))*ln(1.0-(x)))/((1.0+(x))*ln(0.5)+(1.0-(x))*ln(1.0-0.5))}; 
\addplot 
[         domain=-1:1, samples=197,line width=1.2pt, dashed 
] {(-ln(4.0)+(1.0+(x))*ln(1.0+(x))+(1.0-(x))*ln(1.0-(x)))/((1.0+(x))*ln(0.3)+(1.0-(x))*ln(1.0-0.3))}; 
\addplot 
[         domain=-1:1, samples=197,line width=1.2pt,  dotted
] {(-ln(4.0)+(1.0+(x))*ln(1.00000+(x))+(1.0-(x))*ln(1.0000-(x)))/((1.0+(x))*ln(0.95)+(1.0-(x))*ln(1.0-0.95))}; 
\end{axis}
 \end{tikzpicture}\caption{The escape rate spectrum $\alpha\protect\mapsto\delta_{\alpha}(c_{1},c_{2})$
for the random walk model for different values of $c_{1}=0.5$ (solid
line, symmetric case), $c_{1}=0.3$ (dashed line), and $c_{1}=0.9$
(dotted line) and $c_{2}=1-c_{1}$. \label{fig:asymmetric-Random-Walk-Spectrum}}
\end{figure}

Moreover, the Hausdorff dimension of the (uniformly) recurrent set
for the random walk model given by Theorem \ref{thm:Bishop and Jones}
is 
\[
\delta_{0}\left(c_{1},c_{2}\right)=\frac{\log4}{\log(1/c_{1})+\log(1/c_{2})},
\]
see Figure \ref{fig:delta_0Graph} for a one-parameter plot of $c\mapsto\delta_{0}\left(c,1-c\right)$.
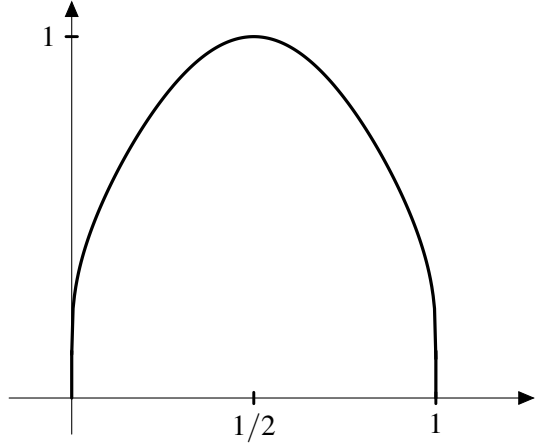
\begin{figure}[h]
\pgfplotsset{width=8.5cm,compat=1.9}
\begin{tikzpicture}[line cap=round,line join=round,>=triangle 45]
\begin{axis}[ymin=-0.1,ymax=1.1,xmin=0,xmax=1.1,     axis equal,     axis lines=middle,     axis line style={->},     tick style={color=black,line width=1.1pt},     xtick={0,0.5,1}, xticklabels={$0$,$1/2$,$1$}, ytick={0,1},yticklabels={$0$, $1$} ]     

\addplot 
[         domain=0:1, samples=270,line width=1.2pt,   
] {-ln(4.0)/ ln(x*(1-x)+0.00001)};
\addplot 
[         domain=0:0.13, samples=5,line width=1.2pt,   
] (1,x);
\addplot [        domain=0:0.13, samples=5,line width=1.2pt,   
] (0,x);
\end{axis}
 \end{tikzpicture}

\caption{The dimension function $c\protect\mapsto\delta_{0}\left(c,1-c\right)$
of the (uniformly) recurrent set for the random walk model parametrised
by contraction rates $\left(c,1-c\right)$ with $c\in\left(0,1\right)$.
A dimension gap does occur for all $c\protect\neq1/2$.\label{fig:delta_0Graph}}
\end{figure}

As an extension of Example \ref{exa:Classical Random Walk} we will
consider asymmetric step widths and interval maps with more than two
branches. The corresponding calculations and precise formulas are
postponed to Section \ref{sec:Examples}.

We end the introduction by providing an application of our results
to the above mentioned $\Z$-extensions of Kleinian groups. More precisely,
we partially recover a result of \cite{rees_1981} and provide a direct
alternative proof, see end of Section \ref{sec:Examples}. Recall
that a group $G$ is of \emph{divergence type} if the series in (\ref{eq:definition Poincare exponent group})
is infinite for the critical exponent $\delta_{G}$.
\begin{cor}
\label{thm:Kleingroup divergence type}If $G=N\rtimes\Z$ is a Schottky
group, then $\delta_{N}=\delta_{G}$ and $N$ is of divergence type.
\end{cor}

\section{Preliminaries and basic definitions\label{sec:Preliminaries-and-basic}}

Let $I$ be a finite set, $I^{\star}\coloneqq\bigcup_{k=1}^{\infty}I^{k}\cup\left\{ \emptyword\right\} $
the set of all finite words over $I$ containing the \emph{empty word}
$\emptyword$ and set $\Sigma\coloneqq I^{\N}$. For $\omega\in I^{\star}$
we denote by $|\omega|$ the unique $k\in\N$ such that $\omega\in I^{k},$
and we refer to $|\omega|$ as the \emph{length} of $\omega$. Note
that $\emptyword$ is the unique word of length zero. For $\omega=\left(\omega_{1},\ldots,\omega_{n}\right)\in I^{n}$
and $1\leq k\leq n$, or $\omega\in\Sigma$ and $k\in\N$ resp., we
set $\omega|_{k}\coloneqq\left(\omega_{1},\ldots,\omega_{k}\right)$.
For $\omega=\left(\omega_{1},\ldots,\omega_{n}\right)\in I^{n}$ and
$\nu=\left(\nu_{1},\ldots,\nu_{k}\right)\in I^{k}$, or $v\in\Sigma$
resp., we define the concatenation $\omega\nu\coloneqq\left(\omega_{1},\ldots,\omega_{n},\nu_{1},\ldots,\nu_{k}\right)$,
or $\omega\nu\coloneqq\left(\omega_{1},\ldots,\omega_{n},\nu_{1},\nu_{2},\ldots\right)$.
We denote by $\sigma:\Sigma\rightarrow\Sigma$ the (left) shift map
given by $\sigma(\omega)_{i}:=\omega_{i+1}$ for every $i\in\N$.
We endow $\Sigma$ with the metric $d(\omega,\tau):=\exp(-\max\left\{ k\ge0\mid\omega_{1}=\tau_{1},\dots,\omega_{k}=\tau_{k}\right\} )$.
In this way we obtain $\left(\Sigma,\sigma\right)$ a continuous dynamical
system over the compact metric space $(\Sigma,d)$.

Recall the definition of a skew-periodic interval map $\FI_{\,\Psi}$
with $\Psi:[0,1]\to\Z$ and of a skew product dynamical system\emph{
}$\left(\Sigma\times\R,\sigma\rtimes f\right)$ with $f:\Sigma\to\R$,
see (\ref{eq:definition skew product dynamical system}).\textcolor{red}{{}
}Let us now give the precise definition of the auxiliary countable
set, where we set $h_{\emptyset}\coloneqq id|_{\left[0,1\right]}$,
\begin{equation}
\mathcal{D}\coloneqq\bigcup_{\left(\omega_{1},\ldots,\omega_{n}\right)\in I^{\star}}h_{\omega_{1}}\circ\cdots\circ h_{\omega_{n}}\left(\left\{ 0,1\right\} \right).\label{eq:D}
\end{equation}
Note that $\pi^{-1}(\mathcal{D})$ is the set of all sequences in
$\Sigma$ which are eventually constant.
\begin{lem}
\label{fac:factor}$\FI_{\,\Psi}$ and $\left.\sigma\rtimes\left(\Psi\circ\pi\right)\right|_{\Sigma\backslash\pi^{-1}(\mathcal{D})\times\Z}$
are topological conjugate.
\end{lem}

\begin{proof}
Let us define the map $\tilde{\pi}:\Sigma\times\Z\to\R$ by $\tilde{\pi}\left(\omega,\ell\right)\coloneqq\pi\left(\omega\right)+\ell$
and restrict its domain to\textcolor{red}{{} }$\widetilde{\pi}^{-1}\left(\mathcal{R}\right)$
(which equals $\Sigma\backslash\pi^{-1}(\mathcal{D})\times\Z$). Then
for $\left(\omega,\ell\right)\in\widetilde{\pi}^{-1}\left(\mathcal{R}\right)$
we have
\[
\FI_{\,\Psi}(\tilde{\pi}\left(\omega,\ell\right))=\sum_{k\in\Z}\left(k+F(\pi\left(\omega\right)+\ell-k)+\Psi(\pi\left(\omega\right)+\ell-k)\right)\1_{[0,1]}(\pi\left(\omega\right)+\ell-k).
\]
Since $F\circ\pi=\pi\circ\sigma$ on $\Sigma\setminus\pi^{-1}\left(\mathcal{D}\right)$,
we have 
\[
\FI_{\,\Psi}(\tilde{\pi}\left(\omega,\ell\right))=\FI(\pi\left(\omega\right))+\ell+\Psi(\pi\left(\omega\right))=\tilde{\pi}(\sigma\rtimes\left(\Psi\circ\pi\right)\left(\omega,\ell\right)).\qedhere
\]
\end{proof}
For $n\in\N$ and $\omega\in I^{\star}$, we set
\begin{equation}
S_{n}f\coloneqq\sum_{k=0}^{n-1}f\circ\sigma^{k},\;S_{0}f\coloneqq0\quad\textnormal{as well as}\quad S_{\omega}f\coloneqq\sup_{x\in\left[\omega\right]}S_{|\omega|}f\left(x\right),\label{eq:birkhoff sum and max sum}
\end{equation}
where $\left[\omega\right]\coloneqq\left\{ x\in\Sigma:x|_{\left|\omega\right|}=\omega\right\} $
denotes the \emph{cylinder set} in $\Sigma$ over $\omega$. We say
$f:\Sigma\rightarrow\R$ is \emph{Hölder continuous} if there exists
$\alpha>0$ such that
\[
\sup_{\omega,\tau\in\Sigma}\left|f(\omega)-f(\tau)\right|/d(\omega,\tau)^{\alpha}<\infty.
\]
Note that by the Hölder continuity there exists a constant $D_{f}$
such that for all $n\in\N$ and $\omega\in I^{n}$ we have 
\begin{equation}
\sup_{x,y\in\left[\omega\right]}\left|S_{n}f\left(x\right)-S_{n}f\left(y\right)\right|\leq D_{f}.\label{eq:boundedDistortion}
\end{equation}
We recall an important result from \cite{MR0419727} adopted to our
situation.
\begin{lem}
\label{lem:nodrift implies recurrence} Suppose $\mu$ is a $\sigma$-invariant
ergodic Borel probability measure and $f:\Sigma\rightarrow\R$ is
$\mu$-integrable. We have
\[
\mu(f)=0\qquad\iff\qquad\liminf_{\ell\to\infty}\left|S_{\ell}f\right|=0\quad\mu\text{-a.e.}
\]
\end{lem}

\begin{proof}
If $\mu$ is non-atomic, we apply the result from \cite{MR0419727}.
For this we consider the natural extension of $\left(\Sigma,\sigma\right)$
which is given by the two-sided shift $I^{\Z}$ making $\sigma:I^{\Z}\to I^{\Z}$
a bi-measurable map with respect to the unique $\sigma$-invariant
measure $\widetilde{\mu}$ on $I^{\Z}$ such that $\mu=\widetilde{\mu}\circ h^{-1}$.
Here, $h$ denotes the canonical projection from $I^{\Z}$ to $\Sigma$.
Then we can transfer the corresponding result from \cite{MR0419727}
for $\widetilde{f}\coloneqq f\circ h$ back to our situation.

If $\mu$ has atoms, then by $\sigma$-invariance and ergodicity we
in fact have $\mu=n^{-1}\sum_{k=0}^{n-1}\delta_{\sigma^{k}x}$, where
$x\in\Sigma$ is some periodic point with period $n$. With this at
hand the equivalence is obvious.
\end{proof}
The following fact puts the previous lemma into the context of infinite
ergodic theory.
\begin{fact*}[{\cite[Theorem 5.5]{MR0578731}}]
Let $\mu$ be a $\sigma$-invariant Borel probability measure and
assume that $f:\Sigma\rightarrow\R$ is $\mu$-integrable. Then the
$(\sigma\rtimes f)$-invariant measure $\mu\times\lambda$ (with $\lambda$
the Lebesgue measure on $\R$) is conservative if and only if $\mu$-a.e.\ $\liminf_{\ell\to\infty}\left|S_{\ell}f\right|=0$.
\end{fact*}
Next, let us recall the classical thermodynamical formalism for full
shifts. On the compact subspace $J^{\N}\subset\Sigma$ the \emph{classical
topological pressure} $\mathfrak{P}(f,J)$ of the continuous \emph{potential}
$f:\Sigma\rightarrow\R$ is given by 
\[
\ \mathfrak{P}(f,J)\coloneqq\lim_{n\rightarrow\infty}\frac{1}{n}\log\sum_{\omega\in J^{n}}\exp\left(S_{\omega}f\right)\qquad\text{and }\qquad\mathfrak{P}(f)\coloneqq\mathfrak{P}(f,I),
\]
in here the limit always exists (see for example \cite{MR648108}).
If $f$ is Hölder continuous and if $J\subset I$ is not a singleton,
then there exists a unique $\sigma$-invariant Borel probability measure
$\mu_{f,J}$ on $J^{\N}$--~called the \emph{Gibbs measure} (for
$f$ restricted to $J^{\N}$)~-- fulfilling the \emph{Gibbs property},
that is, there exists $c\geq1$ such that for all $n\in\N$, $\omega\in J^{n}$
and $x\in[\omega]\cap J^{\N}$ we have
\begin{equation}
c^{-1}\leq\frac{\mu_{f,J}\left(\left[\omega\right]\right)}{\exp(S_{n}f(x)-n\mathfrak{P}(f,J))}\leq c.\label{eq:Gibbs}
\end{equation}
  If $J=\left\{ i\right\} $, then $\mu_{f,J}$ denotes the Dirac
measure $\delta_{x}$ on the constant sequence $x=\left(i,i,\ldots\right)$.
Moreover, for $J=I$ we set $\mu_{f}\coloneqq\mu_{f,I}$.

\section{Fibre-induced Pressure and recurrence properties\label{sec:Fibre-induced-Pressure}}

In this section we will always assume that $\psi,f:\Sigma\rightarrow\R$
are Hölder continuous functions and let 
\begin{equation}
\underline{\psi}\coloneqq\inf_{\omega\in\Sigma}\liminf_{n\rightarrow\infty}S_{n}\psi\left(\omega\right)/n\quad\text{and}\quad\overline{\psi}\coloneqq\sup_{\omega\in\Sigma}\limsup_{n\rightarrow\infty}S_{n}\psi\left(\omega\right)/n.\label{eq:Psiupperlower}
\end{equation}
This gives rise to a skew product dynamical system $\sigma\rtimes\psi:\Sigma\times\R\to\Sigma\times\R$
and we aim at defining the new notion of fibre-induced pressure for
functions $f\circ\pi_{1}:\Sigma\times\R\to\R$ depending only on the
first coordinate where $\pi_{1}:\Sigma\times\R\rightarrow\Sigma$.
Recall that $D_{\text{\ensuremath{\psi}}}$ denotes the distortion
constant of $\psi$ defined in (\ref{eq:boundedDistortion}).

For $K>0$ and $n\in\N$ we define 
\[
\mathcal{C}_{n}(K)\coloneqq\mathcal{C}_{n}\left(\psi,K\right)\coloneqq\left\{ \omega\in I^{n}\mid\left|S_{\omega}\psi\right|\leq K\right\} ,\;\mathcal{C}\left(K\right)\coloneqq\mathcal{C}\left(\psi,K\right)\coloneqq\bigcup_{n\in\N}\mathcal{C}_{n}\left(K\right),
\]
\[
\zeta_{n}(f,\psi,K)\coloneqq\sum_{\omega\in\mathcal{C}_{n}(K)}\e^{S_{\omega}f}
\]
as well as 
\[
\mathcal{P}(f,\psi,K)\coloneqq\limsup_{n\rightarrow\infty}\frac{1}{n}\log\zeta_{n}(f,\psi,K)\quad\text{and }\;\mathcal{P}(f,\psi)\coloneqq\lim_{K\to\infty}\mathcal{P}(f,\psi,K).
\]
We will call $\mathcal{P}(f,\psi)$ the\emph{ fibre-induced pressure}
of $f$ with respect to $\psi.$
\begin{thm}
[Fibre-induced vs.\ classical pressure] \label{thm:fibre-induced pressure via base pressure}Let
$\psi,f:\Sigma\rightarrow\R$ be Hölder continuous functions and fix
$K>D_{\text{\ensuremath{\psi}}}$. If $0\in\big(\underline{\psi},\overline{\psi}\big)$,
then there exists a unique number $t(f)\in\R$ with $\int\psi\,\,d\mu_{t(f)\psi+f}=0$.
For this number we have 
\begin{equation}
\mathcal{P}(f,\psi,K)=\mathfrak{P}(t(f)\psi+f)=\min_{s\in\R}\mathfrak{P}(s\psi+f).\label{eq:InducedPressureViaMinClassicalPressure}
\end{equation}
Further, if we assume that $\psi$ is constant on one-cylinder sets,
allowing also for $0\notin\big(\underline{\psi},\overline{\psi}\big)$,
we have 
\[
\mathcal{P}(f,\psi,K)=\inf_{s\in\R}\mathfrak{P}(s\psi+f)
\]
and the value of the fibre-induced pressure will be finite if and
only if $0\in\big[\underline{\psi},\overline{\psi}\big]$. For $0\in\{\underline{\psi},\overline{\psi}\}$
and setting $I_{0}\coloneqq\left\{ i\in I\colon\psi\left(i,\ldots\right)=0\right\} $
we have in this situation 
\[
\mathcal{P}(f,\psi,K)=\mathfrak{P}(f,I_{0}).
\]
 In any case if $0\in\big(\underline{\psi},\overline{\psi}\big)$
or $\psi$ is constant on one-cylinder sets and $0\in\big[\underline{\psi},\overline{\psi}\big]$,
then we have 
\[
\sum_{\omega\in\mathcal{C}(K)}\e^{S_{\omega}f-|\omega|\mathcal{P}(f,\psi,K)}=\infty.
\]
\end{thm}

\begin{proof}
We first consider the case $0\in\big(\underline{\psi},\overline{\psi}\big)$.
It is well known that the function $H:t\mapsto\mathfrak{P}(t\psi+f)$
is real-analytic and convex (\cite{MR0234697}). By our assumption
on $\psi$, we have that $\lim_{t\rightarrow\pm\infty}H(t)=\infty$.
It follows that $H$ is not affine and hence, $H$ is strictly convex.
We conclude that there is a unique $t(f)\in\R$ such that $H'(t(f))=\int\psi d\mu_{t(f)\psi+f}=0$
and $H(t(f))$ must be the unique minimum of $H$. We have $\mu_{t(f)\psi+f}$-a.e.\ that
$\liminf_{\ell\to\infty}\left|S_{\ell}\psi\right|=0$, by Lemma \ref{lem:nodrift implies recurrence},
and as a consequence of the Borel-Cantelli Lemma we have $\sum_{\ell\ge1}\mu_{t(f)\psi+f}\left\{ \left|S_{\ell}\psi\right|\leq\varepsilon\right\} =\infty$
for every $\varepsilon>0$. This fact combined with the Gibbs property
(\ref{eq:Gibbs}) implies for $\epsilon\coloneqq K-D_{\text{\ensuremath{\psi}}}>0$
and  $c\geq1$ from (\ref{eq:Gibbs}), 
\begin{align*}
\infty=\sum_{\ell\ge1}\mu_{t(f)\psi+f}\left\{ \left|S_{\ell}\psi\right|\leq\varepsilon\right\}  & \leq c\sum_{\omega\in\mathcal{C}(K)}\e^{S_{\omega}\left(f+t\left(f\right)\psi\right)-|\omega|\mathfrak{P}(t(f)\psi+f)}\\
 & \leq c\e^{|t(f)|K}\sum_{\omega\in\mathcal{C}(K)}\e^{S_{\omega}f-|\omega|\mathfrak{P}(t(f)\psi+f)}.
\end{align*}
From this it is easy to see that $\mathcal{P}(f,\psi,K)-\mathfrak{P}(t(f)\psi+f)\ge0$.
Combining this with the obvious estimate $\mathcal{P}(f,\psi,K)\le\mathfrak{P}(s\psi+f)$,
for any $s\in\R$, completes the proof for the case $0\in\big(\underline{\psi},\overline{\psi}\big)$.

As seen above in any case we have $\mathcal{P}(f,\psi,K)\le\inf_{s\in\R}\mathfrak{P}(s\psi+f)$.
Hence, for the boundary cases $0\in\big\{\underline{\psi},\overline{\psi}\big\}$
we are left to verify that $\mathcal{P}(f,\psi,K)\geq\inf_{s\in\R}\mathfrak{P}(s\psi+f)$.
Let us only consider the case $\overline{\psi}=0$ and $\psi$ is
constant on one-cylinders. Then $\inf_{s\in\R}\mathfrak{P}(s\psi+f)=\lim_{s\to+\infty}\mathfrak{P}(s\psi+f)$.
Let us now consider the ergodic invariant Gibbs measure, respectively
atomic measure, $\mu_{f,I_{0}}$. Since $\mu_{f,I_{0}}\left(\psi\right)=0$
and by applying Lemma \ref{lem:nodrift implies recurrence}, we obtain
for $c\geq1$ from (\ref{eq:Gibbs}),
\begin{align*}
\infty=\sum_{\ell\ge1}\mu_{f,I_{0}}\left\{ \left|S_{\ell}\psi\right|\leq K\right\}  & =\adjustlimits\sum_{\ell\ge1}\sum_{\omega\in\mathcal{C}_{\ell}(K)\cap I_{0}^{\ell}}\mu_{f,I_{0}}\left(\left[\omega\right]\right)\\
 & \leq c\adjustlimits\sum_{\ell\ge1}\sum_{\omega\in\mathcal{C}_{\ell}(K)\cap I_{0}^{\ell}}\e^{S_{\omega}f-|\omega|\mathfrak{P}(f,I_{0})}\\
 & \leq c\adjustlimits\sum_{\ell\ge1}\sum_{\omega\in\mathcal{C}_{\ell}(K)}\e^{S_{\omega}f-|\omega|\mathfrak{P}(f,I_{0})}.
\end{align*}
From this it is easy to see that $\mathcal{P}(f,\psi,K)-\mathfrak{P}(f,I_{0})\ge0$.
To see that $\inf_{s\in\R}\mathfrak{P}(s\psi+f)=\mathfrak{P}(f,I_{0})$
we first note that obviously $\mathfrak{P}(s\psi+f)\geq\mathfrak{P}(f,I_{0})$
for any $s\in\R$. By the sub-additivity of the classical pressure,
we have for any $n\in\N$ and $s>0$ that
\[
\mathfrak{P}(s\psi+f)\leq\frac{1}{n}\log\sum_{\omega\in I^{n}}\e^{S_{\omega}\left(s\psi+f\right)}\to\frac{1}{n}\log\sum_{\omega\in I_{0}^{n}}\e^{S_{\omega}f}\;\text{for }s\to\infty.
\]
Hence,
\[
\inf_{s\in\R}\mathfrak{P}(s\psi+f)\leq\frac{1}{n}\log\sum_{\omega\in I_{0}^{n}}\e^{S_{\omega}f}\to\mathfrak{P}(f,I_{0})\quad\text{for }n\to\infty.
\]
In particular, we have $\mathcal{P}(f,\psi,K)>-\infty$ .

Finally, if $0\notin\big[\underline{\psi},\overline{\psi}\big]$,
then $\inf_{s\in\R}\mathfrak{P}(s\psi+f)=-\infty$ and by the obvious
estimate $\mathcal{P}(f,\psi,K)=-\infty=\inf_{s\in\R}\mathfrak{P}(s\psi+f)$.
\end{proof}
\begin{cor}
If $0\in\big(\underline{\psi},\overline{\psi}\big)$ or if $\psi$
is constant on one-cylinder sets, then the fibre-induced pressure
is given by
\[
\mathcal{P}(f,\psi)=\mathcal{P}(f,\psi,K),
\]
for $K>D_{\text{\ensuremath{\psi}}}$.
\end{cor}

We can now characterise when the fibre-induced pressure and classical
pressure coincide.
\begin{cor}
\label{cor:full pressure if no drift} Let $f:\Sigma\rightarrow\R$
and $\psi:\Sigma\rightarrow\R$ be Hölder continuous functions. Then
we have 
\[
\mathcal{P}(f,\psi)=\mathfrak{P}(f)\quad\text{if and only if}\quad\mu_{f}(\psi)=0.
\]
\end{cor}

\begin{proof}
First, assume $\mu_{f}(\psi)=0$. With $H$ defined as in the proof
of Theorem \ref{thm:fibre-induced pressure via base pressure}, we
have $H'\left(0\right)=0$. We distinguish two cases. If $H$ is strictly
convex, then $0\in\big(\underline{\psi},\overline{\psi}\big)$ and
we have $t(f)=0$. Thus, by Theorem \ref{thm:fibre-induced pressure via base pressure},
$\mathcal{P}(f,\psi)=\mathfrak{P}(f)$. If $H$ is affine, then $\psi$
is cohomologous to zero. Hence, also in this case, we have $\mathcal{P}(f,\psi)=\mathfrak{P}(f)$.
Now, suppose $\mu_{f}(\psi)\neq0$. Again by the proof of Theorem
\ref{thm:fibre-induced pressure via base pressure}, $\mathcal{P}(f,\psi)\le\inf_{s\in\R}\mathfrak{P}(s\psi+f)$.
Since $\frac{\partial}{\partial t}\mathfrak{P}(t\psi+f)_{|t=0}=\mu_{f}(\psi)\neq0$,
we conclude $\mathcal{P}(f,\psi)<\mathfrak{P}(f)$.
\end{proof}
For later use we need the following auxiliary statement.
\begin{lem}
\label{lem:strictly decreasing}If $f<0$ and $0\in[\underline{\psi},\overline{\psi}]$,
then $s\mapsto\mathcal{P}(sf,\psi)$ is finite, continuous, strictly
decreasing and we have
\[
\lim_{s\to\pm\infty}\mathcal{P}(sf,\psi)=\mp\infty.
\]
\end{lem}

\begin{proof}
By definition, we have for $K>D_{\text{\ensuremath{\psi}}}$ and $s<t$,
\[
\mathcal{P}(tf,\psi,K)=\limsup_{n\rightarrow\infty}\frac{1}{n}\log\sum_{\omega\in\mathcal{C}_{n}\left(K\right)}\e^{\left((t-s)+s\right)S_{\omega}f}\leq(t-s)\max f+\mathcal{P}(sf,\psi,K).
\]
The claim follows because $\mathcal{P}(sf,\psi,K)\in\R$, $s\in\R$,
which follows from the arguments given in the proof of Theorem \ref{thm:fibre-induced pressure via base pressure}.
\end{proof}
\begin{rem}
\label{rem:Real-Analytic} For $f,\psi$ Hölder continuous functions,
due to the correspondence between the fibre-induced and the classical
pressure as stated in Theorem \ref{thm:fibre-induced pressure via base pressure},
we obtain that $\left(s,a\right)\mapsto\mathcal{P}\left(sf,\psi-a\right)$
is real-analytic with respect to $s\in\R$ and $a\in(\underline{\psi},\overline{\psi})$.
\end{rem}

We end this section introducing the new notion of $\psi$-recurrent
potentials. As already mentioned in the introduction this clarifies
the connection of our setting to the ideas developed in \cite{MR1738951}
and \cite{MR3436756}.

Let $f,\psi:\Sigma\rightarrow\R$ be Hölder continuous. We say that
$f$ is a \emph{$\psi$-recurrent potential} if for all $K>D_{\psi}$
\[
\sum_{\omega\in\mathcal{C}\left(K\right)}\e^{S_{\omega}f-|\omega|\mathcal{P}\left(f,\psi\right)}=\infty.
\]
The following corollary is an immediate consequence of Theorem \ref{thm:fibre-induced pressure via base pressure}.
\begin{cor}
\label{cor:REcurrent} Let $\psi:\Sigma\rightarrow\R$ be a Hölder
continuous function with $\underline{\psi}<0<\overline{\psi}$ or
$\psi$ is constant on one-cylinders. Then any Hölder continuous $f:\Sigma\rightarrow\R$
is a $\psi$-recurrent potential.
\end{cor}

\begin{rem}
\label{rem:original-gurevich} We would like to point out that our
definition of recurrent potentials is analogous to the corresponding
definition for topological Markov chains (\cite{MR1738951}). For
this suppose that $\psi:\Sigma\rightarrow\R$ is constant on $1$-cylinders
such that $\underline{\psi}=\min\psi<0<\max\psi=\overline{\psi}$.
Then we have $D_{\psi}=0$ and $\mathcal{P}(f,\psi)=\mathcal{P}(f,\psi,K)$
for every $K>0$. Moreover, $G:=\left\{ S_{\omega}\psi\mid\omega\in I^{\star}\right\} $
is a countable semi-group and if $G$ is a discrete subset of $\R$,
by Lemma \ref{lem:D} below, $G$ is a countable group. Further, $\sigma\rtimes\psi:\Sigma\times G\rightarrow\Sigma\times G$
is a transitive topological Markov chain with alphabet $I\times G$.
Now, Sarig's definition in \cite{MR1738951} of the Gurevich pressure
of $f\circ\pi_{1}$ with respect to this Markov chain coincides with
our fibre-induced pressure and Sarig's definition of a recurrent potential
also coincides with ours (see also \cite{MR2551790,MR2861747,MR3922537}).
For the closely related concept of induced pressure for Markov shifts
we refer to \cite{MR3190211}.
\end{rem}

\begin{lem}
\label{lem:D}Assume that $\psi:\Sigma\rightarrow\R$ is constant
on one-cylinder sets and $0\in\big(\underline{\psi},\overline{\psi}\big)$.
Then for any $N>0$ and for every $\epsilon>0$ there is a finite
set $\Lambda\subset I^{\star}$ such that for any $\omega\in\mathcal{C}\left(N\right)$
there exists $\nu\in\Lambda$ such that $\omega\nu\in\mathcal{C}\left(\epsilon\right)$.
\end{lem}

\begin{proof}
If $0\in\big(\underline{\psi},\overline{\psi}\big)$, the map $t\mapsto\mathfrak{P}\left(t\psi\right)$
is real-analytic, strictly convex and we have $\lim_{t\to\pm\infty}\mathfrak{P}\left(t\psi\right)=\infty$.
Hence, it has a unique minimum in $t_{0}\in\R$. Then for the unique
Gibbs measure $\mu_{t_{0}\psi}$ with respect to the potential $t_{0}\psi$
we get 
\[
\mu_{t_{0}\psi}\left(\psi\right)=\frac{d}{dt}\mathfrak{P}\left(t\psi\right)|_{t=t_{0}}=0
\]
and this Gibbs measure is also non-atomic and ergodic with respect
to $\sigma$. By Lemma \ref{lem:nodrift implies recurrence}, we obtain
\begin{equation}
\liminf_{\ell\to\infty}\left|S_{\ell}\psi\right|=0\qquad\mu_{t_{0}\psi}\text{-a.e.}\label{eq:recurrence-1}
\end{equation}
For $k=1,\ldots,m\coloneqq\left\lceil 2N/\epsilon\right\rceil $ we
set 
\[
I_{k}\coloneqq\left(\left(k-1\right)\epsilon/2,k\epsilon/2\right]\quad\text{and}\quad I_{-k}\coloneqq\left[-k\epsilon/2,\left(-k+1\right)\epsilon/2\right).
\]
We can ignore all $k\in\left\{ \pm1,\ldots,\pm m\right\} $ such that
$S_{\omega}\psi\notin I_{k}$ for all $\omega\in I^{\star}$. For
the remaining $k$ we find $\omega_{k}\in I^{\star}$ such that $S_{\omega_{k}}\psi\in I_{k}$.
Since every non-empty cylinder $[\omega]$ with $\omega\in I^{\star}$
carries\textcolor{red}{{} }positive $\mu_{t_{0}\psi}$-measure, using
(\ref{eq:recurrence-1}) we find $x\in\left[\omega_{k}\right]$ and
$n\geq\left|\omega_{k}\right|$ such that $\left|S_{n}\psi\left(x\right)\right|<\epsilon/2$.
With $\nu_{k}\coloneqq\left(x_{\left|\omega_{k}\right|+1},\ldots,x_{n}\right)$
we have $\left|S_{\omega_{k}\nu_{k}}\psi\right|<\epsilon/2$ applying
(\ref{eq:boundedDistortion}). Hence using (\ref{eq:boundedDistortion})
once more, for an arbitrary $\omega\in I^{\star}$ with $S_{\omega}\psi\in I_{k}$,
we have for all $x\in\left[\omega\nu_{k}\right]$
\begin{align*}
\left|S_{\left|\omega\nu_{k}\right|}\psi\left(x\right)\right| & =\left|S_{\left|\omega\nu_{k}\right|}\psi\left(x\right)-S_{\left|\omega_{k}\nu_{k}\right|}\psi\left(\omega_{k}\sigma^{\left|\omega\right|}\left(x\right)\right)+S_{\left|\omega_{k}\nu_{k}\right|}\psi\left(\omega_{k}\sigma^{\left|\omega\right|}\left(x\right)\right)\right|\\
 & \leq\left|S_{\left|\omega\nu_{k}\right|}\psi\left(x\right)-S_{\left|\omega_{k}\nu_{k}\right|}\psi\left(\omega_{k}\sigma^{\left|\omega\right|}\left(x\right)\right)\right|+\left|S_{\left|\omega_{k}\nu_{k}\right|}\psi\left(\omega_{k}\sigma^{\left|\omega\right|}\left(x\right)\right)\right|,\\
 & \leq\left|S_{\left|\omega\right|}\psi\left(x\right)-S_{\left|\omega_{k}\right|}\psi\left(\omega_{k}\sigma^{\left|\omega\right|}\left(x\right)\right)\right|+\epsilon/2\\
 & \leq\epsilon/2+\epsilon/2=\epsilon
\end{align*}
proving the desired statement.
\end{proof}
\begin{rem}
\label{rem:characterisation of recurrence}For a subgroup $G<\Z$
and $\psi:\Sigma\rightarrow G$ it is shown in \cite{MR3436756} that
$f\circ\pi_{1}$ is recurrent (according to Sarig, see \cite{MR1738951})
if and only if there is a character $c:G\rightarrow(0,\infty)$ such
that $\mu_{f-\log c\circ\psi}\times\lambda_{G}$ is the conservative
equilibrium measure of $f\circ\pi_{1}$, where $\lambda_{G}$ denotes
the counting measure on $G$ (for the definition of an equilibrium
measure in this setting, see \cite{MR1818392}). If $\psi$ is constant
on one-cylinders and $0\in\big(\underline{\psi},\overline{\psi}\big)$,
then $f-\log c\circ\psi$ is of the form $f+t(f)\psi$ as stated in
Theorem \ref{thm:fibre-induced pressure via base pressure}. Further,
we can derive (\ref{eq:InducedPressureViaMinClassicalPressure}) from
\cite[Theorem 1(1)]{MR3436756} where we note again that the fibre-induced
and Gurevich pressure coincide in this setting, see the previous remark.
Moreover, by \cite[ Proposition 1.4]{MR3436756} we have that $\mathcal{P}(f,\psi)=\mathfrak{P}(f)$
if and only if $f$ is \emph{symmetric on average}, i.e.
\[
\sup_{m\in\Z}\limsup_{n\rightarrow\infty}\frac{\sum_{\left|\omega\right|\le n,S_{\omega}\psi=m}\e^{S_{\omega}f-\left|\omega\right|\mathcal{P}\left(f,\psi\right)}}{\sum_{\left|\omega\right|\le n,S_{\omega}\psi=-m}\e^{S_{\omega}f-\left|\omega\right|\mathcal{P}\left(f,\psi\right)}}<\infty.
\]
 By Corollary \ref{cor:full pressure if no drift} this condition
is in fact equivalent to $\mu_{f}(\psi)=0$.
\end{rem}

\section{Multifractal decomposition with respect to $\alpha$-escaping sets\label{sec:Multifractal-decomposition}}

In this section $\psi:\Sigma\to\Z$ denotes the symbolic step length
function which is constant on one-cylinder sets and such that $\psi=\Psi\circ\pi$
except possibly on a finite set. Note that under these assumption
on $\psi$ we have $D_{\psi}=0$. For a given drift parameter $\alpha\in\big[\underline{\psi},\overline{\psi}\big]$
let us set 
\[
\psi_{\alpha}:=\psi-\alpha.
\]
Note that we have
\begin{align*}
\mathbf{E}(\alpha) & \cap\left[0,1\right]=\pi\left\{ \omega\in\Sigma\mid\exists K>0\;\left|S_{n}\psi_{\alpha}(\omega)\right|\leq K\;\text{for infinitely many }n\in\N\right\} \setminus\mathcal{D}
\end{align*}
and 
\begin{align*}
\mathbf{E}_{u}(\alpha)\cap\left[0,1\right] & =\pi\left\{ \omega\in\Sigma\mid\exists K>0\;\forall n\in\N\;\left|S_{n}\psi_{\alpha}(\omega)\right|\leq K\right\} \setminus\mathcal{D},
\end{align*}
for the countable set $\mathcal{D}$, see (\ref{eq:D}). Recall the
definition of the $\alpha$-Poincaré exponent $\delta_{\alpha}$ with
$K>\D$, see (\ref{eq:definition Poincare exponent}). Next we give
the proof of the implication (\ref{eq:delta_alpha=00003D0}) stated
in the introduction.
\begin{proof}
[Proof of implication (\ref{eq:delta_alpha=00003D0})] For $\alpha\in\R\setminus\big[\underline{\psi},\overline{\psi}\big]$
we have that $S_{n}\psi_{\alpha}$ diverges uniformly to either $+\infty$
or $-\infty$ and hence, being a finite sum, $\sum_{{\omega\in I^{\star}},\,{\left|S_{\omega}\left(\psi-\alpha\right)\right|\leq K}}\e^{s\cdot S_{\omega}\varphi}$
is finite for all $s\geq0$.
\end{proof}
The following assertion will be crucial for determining the multifractal
decomposition with respect to our escaping sets, and to derive explicit
formulas for dimension gaps.
\begin{thm}
\label{thm:critical exponent via base pressures} We have the following
three different characterisations of the $\alpha$-Poincaré exponent
$\delta_{\alpha}$:
\begin{enumerate}
\item[(i)] \label{enu:-delta_Alpha_Pressure_new} For $\alpha\in\big[\underline{\psi},\overline{\psi}\big]$
we have that $\delta_{\alpha}$ is uniquely determined by
\[
{\displaystyle \mathcal{P}\left(\delta_{\alpha}\varphi,\psi_{\alpha}\right)=0.}
\]
\item[(ii)] \label{enu:There-exist-a(ii)}For $\alpha\in\big(\underline{\psi},\overline{\psi}\big)$
we have that $\delta_{\alpha}$ is determined by the unique solution
$\left(\delta_{\alpha},q_{\alpha}\right)\in\R^{2}$ of the equations
\[
{\displaystyle \mathfrak{P}\left(\delta_{\alpha}\varphi+q_{\alpha}\psi_{\alpha}\right)=0\quad\text{and}\quad\frac{\partial}{\partial q}\mathfrak{P}(\delta_{\alpha}\varphi+q\psi_{\alpha})|_{q=q_{\alpha}}=0}.
\]
In particular, we have $\int\psi\,d\mu_{\delta_{\alpha}\varphi+q_{\alpha}\psi_{\alpha}}=\alpha$.
\item[(iii)] \label{enu:delta_alpha via MF-Formalism}Let $s:\R\times\R\to\R$
be the implicitly defined function given by $\mathfrak{P}\left(s\left(q,a\right)\varphi+q\psi+a\1\right)=0$.
Then $s$ is convex and for the \emph{Legendre transform} of $s$
defined by
\[
\widehat{s}\left(\alpha_{1},\alpha_{2}\right)\coloneqq\sup_{(q,a)\in\R^{2}}q\alpha_{1}+a\alpha_{2}-s\left(q,a\right)
\]
we have for $\alpha\in\big(\underline{\psi},\overline{\psi}\big)$
\[
\delta_{\alpha}=-\inf_{r\in\R}\widehat{s}\left(\alpha r,r\right)=\inf_{r\in\R}s\left(r,-\alpha r\right).
\]
In particular, we have that on the line $g$ through the origin with
slope $-\alpha$ in the $\left(q,a\right)$-plane there exists exactly
one point $\left(r,-\alpha r\right)$ such that the gradient of $s$
in this point is perpendicular to $g$. The height of $s$ in $\left(r,-\alpha r\right)$
is $\delta_{\alpha}$ and the plane which is tangential to the graph
of $s$ in $\left(r,-\alpha r,s\left(r,-\alpha r\right)\right)$ intersects
the $s$-axis in $-\delta_{\alpha}$ (cf.\ Figure \ref{fig:Graph-of-1}).
\begin{figure}
\includegraphics[width=0.5\textwidth]{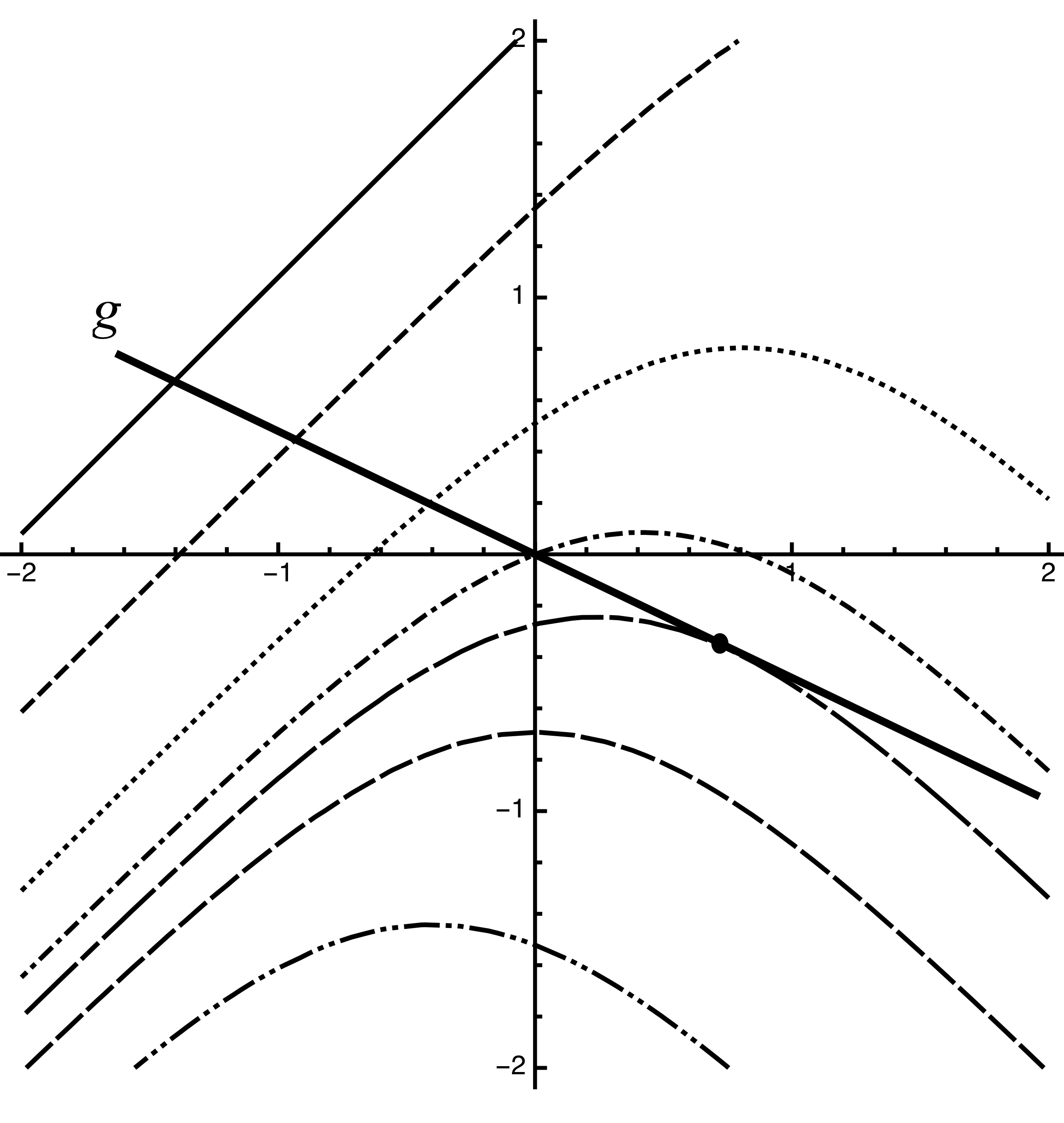}

\caption{\label{fig:Graph-of-1}The contour plot for the random walk model
with $c_{1}=0.1$, $c_{2}=0.5$ of $s\left(q,a\right)$ which is implicitly
defined by $\mathfrak{P}\left(s\left(q,a\right)\varphi+q\psi+a\1\right)=0$.
The fourth contour line is determined by $s\left(q,a(q)\right)=1$,
i.e.\ corresponds to height $\delta$ (defined in (\ref{eq:Hausdorff dim Bowen formula})).
The line $g$ goes through the origin with slope $-\alpha$.}
\end{figure}
\end{enumerate}
\end{thm}

\begin{proof}
(i) To prove that the $\alpha$-Poincaré exponent coincides with the
zero of the fibre-induced pressure first note that for all $\alpha\in\big[\underline{\psi},\overline{\psi}\big]$
we have that there is a unique zero $\tilde{\delta}_{\alpha}\in\R$
of the function $s\mapsto\mathcal{P}\left(s\varphi,\psi_{\alpha}\right)$,
by Lemma \ref{lem:strictly decreasing}. By definition, we have $\mathcal{P}\left(s\varphi,\psi_{\alpha}\right)=\limsup_{n\to\infty}n^{-1}\log\sum_{{\omega\in\mathcal{C}_{n}\left(\psi_{\alpha},K\right)}}\e^{sS_{\omega}\varphi}<0$
for $s>\tilde{\delta}_{\alpha}$ and $K>0$. Hence, 
\[
\sum_{{\omega\in I^{\star}},\,{\left|S_{\omega}\left(\psi-\alpha\right)\right|\leq K}}\e^{sS_{\omega}\varphi}=\sum_{n\in\N}\sum_{{\omega\in\mathcal{C}_{n}\left(\psi_{\alpha},K\right)}}\e^{sS_{\omega}\varphi}=\sum_{n\in\N}\exp\left(n\cdot\frac{1}{n}\log\sum_{{\omega\in\mathcal{C}_{n}\left(\psi_{\alpha},K\right)}}\e^{sS_{\omega}\varphi}\right)<\infty
\]
implying $s\geq\delta_{\alpha}$, and by the arbitrariness of $s>\tilde{\delta}_{\alpha}$
we have $\tilde{\delta}_{\alpha}\geq\delta_{\alpha}.$ For the reverse
inequality assume $s<\tilde{\delta}_{\alpha}$ . Then $\mathcal{P}\left(s\varphi,\psi_{\alpha}\right)>0$.
Again by definition of the pressure, there exists a sequence $n_{j}\nearrow\infty$
such that $n_{j}^{-1}\log\sum_{{\omega\in\mathcal{C}_{n_{j}}\left(\psi_{\alpha},K\right)}}\e^{sS_{\omega}\varphi}>0$
for all $j\in\N$. Hence, 
\[
\sum_{{\omega\in I^{\star}},\,{\left|S_{\omega}\left(\psi-\alpha\right)\right|\leq K}}\e^{sS_{\omega}\varphi}=\sum_{n\in\N}\sum_{{\omega\in\mathcal{C}_{n}\left(\psi_{\alpha},K\right)}}\e^{sS_{\omega}\varphi}\geq\sum_{j\in\N}\exp\left(0\right)=\infty.
\]
This gives $s\leq\delta_{\alpha}$. Since this holds for every $s<\tilde{\delta}_{\alpha}$,
we conclude $\tilde{\delta}_{\alpha}\leq\delta_{\alpha}$.

(ii) By the Hölder continuity of $\varphi$ and $\psi$ we have that
the function $p:(s,t)\mapsto\text{\ensuremath{\mathfrak{P}}}\left(s\varphi+t\psi_{\alpha}\right)$
is convex and real-analytic with respect to both coordinates. For
fixed $q\in\R$ the function $s\mapsto p(s,q)$ is strictly monotonically
decreasing from $+\infty$ to $-\infty$, hence there is a unique
number $s\left(q\right)$ such that $p\left(s\left(q\right),q\right)=0$.
Also the function $s:q\mapsto s\left(q\right)$ is real-analytic and
convex by the Implicit Function Theorem and is tending to infinity
for $q\to\pm\infty$. Hence, there is a unique number $q_{\alpha}\in\R$
minimising $s(q)$, or equivalently, such that $s'\left(q_{\alpha}\right)=0=-\frac{\partial}{\partial q}p\left(s(q_{\alpha}),q_{\alpha}\right)=\int\psi_{\alpha}\,d\mu_{s(q_{\alpha})\varphi+q_{\alpha}\psi_{\alpha}}$.
To see that $\delta_{\alpha}$ in fact coincides with $s\left(q_{\alpha}\right)$
we combine this observation, our first alternative characterisation
of $\delta_{\alpha}$ and Theorem \ref{thm:fibre-induced pressure via base pressure}
(with $f=s\left(q_{\alpha}\right)\varphi$) to get
\[
\mathcal{P}(s\left(q_{\alpha}\right)\varphi,\psi_{\alpha})=\mathfrak{P}\left(s\left(q_{\alpha}\right)\varphi+q_{\alpha}\psi_{\alpha}\right)=0.
\]
Obviously, by the definition of $\psi_{\alpha}$ we also find $\int\psi\,d\mu_{s(q_{\alpha})\varphi+q_{\alpha}\psi_{\alpha}}=\alpha$.

(iii) For the last alternative characterisation we observe that the
contour lines of $s$ of height $d$ can be given explicitly by 
\[
a_{d}\left(q\right)\coloneqq-\mathfrak{P}\left(d\varphi+q\psi\right)\quad\text{with }a_{d}'\left(q\right)=-\int\psi\,d\mu_{d\varphi+q\psi}.
\]
Given $\alpha\in\big(\underline{\psi},\overline{\psi}\big)$, we find
that for $d=\delta_{\alpha}$ and $q=q_{\alpha}$ we have by (ii)
that the line through the origin with slope $-\alpha$ is tangential
to the contour line $a_{\delta_{\alpha}}$ of height $\delta_{\alpha}$
in the point $\left(q_{\alpha},-\alpha q_{\alpha}\right)$ (see Figure
\ref{fig:Graph-of-1}). This shows that 
\[
\delta_{\alpha}=\inf_{r\in\R}s\left(r,-\alpha r\right).
\]
Now, the gradient of $s$ in $\left(q_{\alpha},-\alpha q_{\alpha}\right)$
is $\left(\alpha t_{\alpha},t_{\alpha}\right)$ with $t_{\alpha}\coloneqq1/\int\varphi\,d\mu_{\delta_{\alpha}\varphi+q_{\alpha}\psi_{\alpha}}$
and hence, for all $t\in\R$ 
\[
s\left(q_{\alpha},-\alpha q_{\alpha}\right)\geq-\widehat{s}\left(\alpha t,t\right),
\]
with equality if $t=t_{\alpha}$ (see e.g.\ \cite[Theorem 23.5]{rockafellar1970}).
This proves $\delta_{\alpha}=\sup_{t\in\R}-\widehat{s}\left(\alpha t,t\right)$.
\end{proof}
\begin{rem}
\label{rem:delta_=00005Calpha as Birkhoff average}As a consequence
of Remark \ref{rem:Real-Analytic}, the fact that $\varphi<0$, and
the Implicit Function Theorem, we have that the mapping $\alpha\mapsto\delta_{\alpha}$
is real-analytic on $\big(\underline{\psi},\overline{\psi}\big)$.
Also note that the characterisations in Theorem \ref{thm:critical exponent via base pressures}
in particular show that we have the following multifractal spectral
identity
\[
\delta_{\alpha}=\dim_{H}\left\{ \pi\left(\omega\right)\colon\omega\in\Sigma,\,\lim_{n\rightarrow\infty}\frac{S_{n}(\psi\circ\pi)\left(\omega\right)}{n}=\alpha\right\} .
\]
See also \cite{MR1858487,MR1837214,MR1916371,MR3286501,JordanRams2017}
and Remark \ref{rem:Multifractal_Alternative_Proof}.
\end{rem}

We are now in the position to prove Theorems \ref{thm:-dimension gap}
and \ref{thm: Multifractal Decomposition} stated in the introduction.
\begin{proof}
[Proof of Theorem \ref{thm:-dimension gap}]By definition of $\delta$
and Theorem \ref{thm:critical exponent via base pressures}, we have
$\delta_{\alpha}=\delta>0$ if and only if $\mathfrak{P}(\delta\varphi)=0=\mathcal{P}\left(\delta\varphi,\psi_{\alpha}\right)$.
Hence by Corollary \ref{cor:full pressure if no drift}, this holds
if and only if $\mu_{\delta\varphi}(\psi_{\alpha})=0$ which is equivalent
to $\mu_{\delta\varphi}\left(\psi\right)=\alpha$.
\end{proof}
Before we proceed with the proof of Theorem \ref{thm: Multifractal Decomposition},
let us fix some further properties of $\alpha\mapsto\delta_{\alpha}$
in the following proposition.
\begin{prop}
\label{prop:Properties of delta_alpha}If $\underline{\psi}<\overline{\psi}$,
then the function $\alpha\mapsto\delta_{\alpha}$ given by the $\alpha$-Poincaré
exponent is strictly increasing on $\left(\underline{\psi},\alpha_{\text{max}}\right)$
and strictly decreasing on $\left(\alpha_{\text{max}},\overline{\psi}\right)$
and hence has a unique maximum in $\alpha_{\text{max }}\coloneqq\mu_{\delta\varphi}\left(\psi\right)$
with value $\delta$. Further, we have 
\[
\frac{d}{d\alpha}\delta_{\alpha}=\frac{q_{\alpha}}{\mu_{\delta_{\alpha}\varphi+q_{\alpha}\psi_{\alpha}}\left(\varphi\right)}
\]
and in particular, 
\[
\lim_{\alpha\to\overline{\psi}}\frac{d}{d\alpha}\delta_{\alpha}=-\infty\quad\text{and}\quad\lim_{\alpha\to\underline{\psi}}\frac{d}{d\alpha}\delta_{\alpha}=+\infty.
\]
\end{prop}

\begin{proof}
First, for $\alpha\in\big(\underline{\psi},\overline{\psi}\big)$,
we check that with $\mu=\mu_{\delta_{\alpha}\varphi+q_{\alpha}\psi_{\alpha}}$,
\[
\frac{\partial}{\partial\alpha}\left(\mathfrak{P}(\delta_{\alpha}\varphi+q_{\alpha}\psi_{\alpha})\right)=\frac{d}{d\alpha}\delta_{\alpha}\int\varphi d\mu+\frac{d}{d\alpha}q_{\alpha}\int\psi d\mu-\frac{d}{d\alpha}q_{\alpha}\alpha-q_{\alpha}.
\]
By Theorem \ref{thm:critical exponent via base pressures}, we conclude
that $\frac{\partial}{\partial\alpha}\left(\mathfrak{P}(\delta_{\alpha}\varphi+q_{\alpha}\psi_{\alpha})\right)=0$.
Hence, $q_{\alpha}=\frac{d}{d\alpha}\delta_{\alpha}\int\varphi d\mu$.
We are left to prove the monotonicity. We find that the contour lines
of $s$ given by the function $q\mapsto a_{c}\left(q\right)$ implicitly
via $s\left(q,a_{c}(q)\right)=c$ intersect the $a$-axis in $-\mathfrak{P}\left(c\varphi\right)$
for all $c\in\R$. Since $a_{\delta}\left(0\right)=0$, $a_{\delta}'\left(0\right)=-\alpha_{\text{max}}$
and $-\mathfrak{P}\left(c\varphi\right)<0$ for $c<\delta$, we find
-- using the characterisation of $\delta_{\alpha}$ from Theorem
\ref{thm:critical exponent via base pressures} (iii) -- that $q_{\alpha}<0$
for $\alpha\in\left(\underline{\psi},\alpha_{\text{max}}\right)$
and $q_{\alpha}>0$ for $\alpha\in\left(\alpha_{\text{max}},\overline{\psi}\right)$.
\end{proof}
\begin{proof}
[Proof of Theorem  \ref{thm: Multifractal Decomposition}] For the
upper bound we use a straight forward covering argument and the definition
of the $\alpha$-Poincaré exponent. The lower bound uses the fibre-induced
pressure and a concrete construction of a Frostman measure inspired
by \cite{MR1484767}. Let 
\[
L_{n}(\alpha,k):=\bigcup_{\omega\in\mathcal{C}_{n}\left(\psi_{\alpha},k\right)}\pi\left[\omega\right],
\]
for $k\in\N$ and set 
\[
L(\alpha,k)\coloneqq\limsup_{n\rightarrow\infty}L_{n}(\alpha,k)=\bigcap_{n\in\N}\bigcup_{m\geq n}L_{m}(\alpha,k).
\]
Clearly, $\mathbf{E}\left(\alpha\right)\subset\bigcup_{k\in\N}L(\alpha,k)$.
By the stability of the Hausdorff dimension it suffices to prove that
$\dim_{H}\left(L(\alpha,k)\right)\le\delta_{\alpha}$ for every $k\in\N$.
For every $s>\delta_{\alpha}$, we have by the definition of the Poincaré
exponent for any $N\in\N$ 
\[
\sum_{n\geq N}\sum_{\omega\in\mathcal{C}_{n}\left(\psi_{\alpha},k\right)}\left|\pi\left[\omega\right]\right|^{s}\leq\sum_{n\geq N}\sum_{\omega\in\mathcal{C}_{n}\left(\psi_{\alpha},k\right)}\e^{S_{\omega}s\varphi}\leq\sum_{\omega\in\mathcal{C}\left(\psi_{\alpha},k\right)}\e^{S_{\omega}s\varphi}<\infty.
\]
This shows that the $s$-dimensional Hausdorff measure of $L\left(\alpha,k\right)$
is finite and hence $\dim_{H}(L(\alpha,k))\le s$ for every $s>\delta_{\alpha}$
giving the upper bound.

For the lower bound we start with the case $\alpha\in\big\{\underline{\psi},\overline{\psi}\big\}$.
For $I_{0}\coloneqq\left\{ i\in I\colon\psi_{\alpha}\left(i,\ldots\right)=0\right\} $
we clearly have $\pi\left(I_{0}^{\N}\right)\subset\mathbf{E}_{u}\left(\alpha\right)$.
By Theorem \ref{thm:fibre-induced pressure via base pressure}, we
have $0=\mathcal{P}\left(\delta_{\alpha}\varphi,\psi_{\alpha}\right)=\mathfrak{P}(\delta_{\alpha}\varphi,I_{0})$
and hence by Bowen's formula $\dim_{H}\left(\pi\left(I_{0}^{\N}\right)\right)=\delta_{\alpha}$
providing the lower bound for this case. For $\alpha\in\big(\underline{\psi},\overline{\psi}\big)$
we consideronly the case $\delta_{\alpha}>0$. It then suffices to
construct for every positive $s<\delta_{\alpha}$ a subset $C\subset\mathbf{E}_{u}\left(\alpha\right)$
such that $\dim_{H}C>s$. Fix such positive $s<\delta_{\alpha}$ and
let $\epsilon>0$. Then $\mathcal{P}\left(s\varphi,\psi_{\alpha}\right)>0$
by Theorem \ref{thm:critical exponent via base pressures} (i). Hence,
for $\epsilon>0$, 
\[
\limsup_{n\to\infty}\sum_{\omega\in\mathcal{C}_{n}\left(\psi_{\alpha},\epsilon\right)}\e^{sS_{\omega}\varphi}=\infty.
\]
Hence, for arbitrarily large $M>0$ we find $\ell\in\N$ with $\sum_{\omega\in\mathcal{C}_{\ell}\left(\psi_{\alpha},\epsilon\right)}\e^{sS_{\omega}\varphi}>M$.
Set $\Gamma:=\mathcal{C}_{\ell}\left(\psi_{\alpha},\epsilon\right)$.
Further, let $\Lambda\subset I^{\star}$ denote the finite set of
words of length less than $r\in\N$ witnessing the conditions stated
in Lemma \ref{lem:D} for $N=2\epsilon$. Set $\Gamma_{1}\coloneqq\Gamma$.
Then for every $\omega\in\Gamma_{1}$ we have $\left|S_{\omega}\psi_{\alpha}\right|\leq\epsilon$.
Suppose we have defined $\Gamma_{n}$ such that $\left|S_{\omega}\psi_{\alpha}\right|\leq\epsilon$
for all $\omega\in\Gamma_{n}$. Then we define inductively $\Gamma_{n+1}$
as follows: Since for $\omega\in\Gamma_{n}$ we have $\left|S_{\omega}\psi_{\alpha}\right|\leq\epsilon$
it follows that $\left|S_{\omega\nu}\psi_{\alpha}\right|\leq2\epsilon$
for every $\nu\in\Gamma$. Then we find an element $\tau_{\omega\nu}\in\Lambda$
such that $\left|S_{\omega\nu\tau_{\omega\nu}}\psi_{\alpha}\right|\leq\epsilon$
and we set
\[
\Gamma_{n+1}\coloneqq\left\{ \omega\nu\tau_{\omega\nu}:\omega\in\Gamma_{n},\nu\in\Gamma\right\} .
\]
Next, using Hölder continuity of $\varphi$ and bounded distortion,
we find suitable constants $c_{i}$, $i=1,\ldots,5$, such that for
all $\omega\in\Gamma_{n}$, $n\in\N$, 
\begin{align}
\sum_{\nu\tau_{\omega\nu}:\omega\nu\tau_{\omega\nu}\in\Gamma_{n+1}}\left|\pi\left[\omega\nu\tau_{\omega\nu}\right]\right|^{s} & \geq c_{1}\sum_{\nu\tau_{\omega\nu}:\omega\nu\tau_{\omega\nu}\in\Gamma_{n+1}}\e^{sS_{\omega\nu\tau_{\omega\nu}}\varphi}\geq c_{2}\e^{sS_{\omega}\varphi}\sum_{\nu\tau_{\omega\nu}:\omega\nu\tau_{\omega\nu}\in\Gamma_{n+1}}\e^{sS_{\nu\tau_{\omega\nu}}\varphi}\nonumber \\
\nonumber \\
 & \geq c_{3}\e^{sS_{\omega}\varphi}\sum_{\gamma\in\Gamma}\e^{sS_{\gamma}\varphi}\geq c_{4}M\e^{sS_{\omega}\varphi}\geq c_{5}M\left|\pi\left[\omega\right]\right|^{s}\geq\left|\pi\left[\omega\right]\right|^{s},\label{eq:GeometricInequality}
\end{align}
where the last inequality holds for $\ell$, and hence for $M$, chosen
sufficiently large.

In the next step we define the Cantor set 
\[
C\coloneqq\limsup_{n\rightarrow\infty}\bigcup_{\omega\in\Gamma_{n}}\pi\left[\omega\right]
\]
and a probability measure $\mu$ supported on $C$ defined by its
marginals as follows: $\mu\left(\left[0,1\right]\right)=1$ and for
$\omega\in\Gamma_{n}$ and $\omega\nu\tau_{\omega\nu}\in\Gamma_{n+1}$
\[
\mu\left(\pi\left[\omega\nu\tau_{\omega\nu}\right]\right)\coloneqq\frac{\left|\pi\left[\omega\nu\tau_{\omega\nu}\right]\right|^{s}}{\sum_{\widetilde{\omega}\in\Gamma_{n+1},\left[\widetilde{\omega}\right]\subset\left[\omega\right]}\left|\pi\left[\widetilde{\omega}\right]\right|^{s}}\mu\left(\pi\left[\omega\right]\right).
\]
Note that the existence of $\mu$ is guaranteed by Kolmogorov's Consistency
Theorem. Inductively, using the definition of $\mu$ in tandem with
inequality (\ref{eq:GeometricInequality}) we verify that for all
$n\in\N$ and $\omega\in\Gamma_{n}$ we have 
\[
\mu\left(\pi\left[\omega\right]\right)\leq\left|\pi\left[\omega\right]\right|^{s}.
\]
For each interval $L\subset\left[0,1\right]$ we let 
\begin{multline*}
\Gamma\left(L\right)\coloneqq\bigg\{\omega\in\Gamma_{n}\colon n\in\N,\;\pi\left[\omega\right]\cap L\neq\emptyset,\;\left|\pi\left[\omega\right]\right|\text{\ensuremath{\leq}}\left|L\right|\text{ and }\widetilde{\omega}\in\Gamma_{n-1}:\left[\widetilde{\omega}\right]\supset\left[\omega\right]\Rightarrow\left|\pi\left[\widetilde{\omega}\right]\right|>\left|L\right|\bigg\}.
\end{multline*}
Since 
\[
\mu\left(L\right)\leq\sum_{\omega\in\Gamma\left(L\right)}\mu\left(\pi\left[\omega\right]\right)\leq\sum_{\omega\in\Gamma\left(L\right)}\left|\pi\left[\omega\right]\right|^{s}\leq\card\left(\Gamma\left(L\right)\right)\left|L\right|^{s}
\]
and $\card\left(\Gamma\left(L\right)\right)\leq2\cdot\card\left(I\right)^{\ell+r}$,
the Mass Distribution Principle gives 
\[
\dim_{H}\left(C\right)\geq s.
\]
Since $s<\delta_{\alpha}$ was arbitrary,\textcolor{red}{{} }and $\mu\left(C\right)=\mu\left(C\setminus\mathcal{D}\right)$
the desired lower bound is established.
\end{proof}
\begin{rem}
\label{rem:Multifractal_Alternative_Proof} Let $\mu_{\alpha}:=\mu_{\delta_{\alpha}\varphi+q\psi_{\alpha}}\circ\pi^{-1}$.
Since $\mu_{\delta_{\alpha}\varphi+q\psi_{\alpha}}(\psi_{\alpha})=0$,
Lemma \ref{lem:nodrift implies recurrence} implies that $\mu_{\alpha}\left(\mathbf{E}(\alpha)\right)=1$.
By Young's formula we have $\dim_{H}(\mu_{\alpha})=\delta_{\alpha}$,
which then gives an alternative proof for the lower bound of the Hausdorff
dimension of $\mathbf{E}\left(\alpha\right)$.
\end{rem}

\begin{proof}
[Proof of Theorem  \ref{thm:Bishop and Jones}] First suppose that
$0\in\big[\underline{\psi},\overline{\psi}\big]$. Then this corollary
is a special case of Theorem \ref{thm: Multifractal Decomposition}
with $\alpha=0$ by noting that $\mathbf{R}=\mathbf{E}\left(0\right)$
and $\mathbf{R}_{u}=\mathbf{E}_{u}\left(0\right)$. If $0\notin\big[\underline{\psi},\overline{\psi}\big]$,
then $\mathbf{R}=\mathbf{E}\left(0\right)=\mathbf{R}_{u}=\mathbf{E}_{u}\left(0\right)=\emptyset$
and $\delta_{0}=0$.
\end{proof}
\begin{proof}
[Proof of Theorem  \ref{thm:-Transient-sets}]We only consider the
first case, the second case follows in exactly the same manner. For
this we assume $\alpha=\mu_{\delta\varphi}\left(\psi\right)\ge0$.
Since $\mathbf{E}_{u}\left(\alpha'\right)\subset\mathbf{T}_{1}^{+}\subset\mathbf{T}_{2}^{+}$
for all $\alpha'>0$, we obtain 
\begin{align*}
\delta & =\dim_{H}\left(\mathbf{E}_{u}\left(\alpha\right)\right)=\sup_{\alpha'>\alpha}\dim_{H}\left(\mathbf{E}_{u}\left(\alpha'\right)\right)\leq\dim_{H}\left(\mathbf{T}_{1}^{+}\right)\leq\dim_{H}\left(\mathbf{T}_{2}^{+}\right)\leq\delta.
\end{align*}
For the set $\mathbf{T}_{3}^{+}\left(r\right)$, $r\in\R$ note that
\[
\bigcup_{\ell\geq0}\FI_{\Psi}^{-\ell}\left(\mathbf{T}_{3}^{+}\left(r\right)\right)\supset\mathbf{T}_{1}^{+}.
\]
Since for each $\ell\in\N$ the set $F_{\Psi}^{-\ell}\left(\mathbf{T}_{3}^{+}\left(r\right)\right)$
is a countable union of bi-Lipschitz images of $\mathbf{T}_{3}^{+}\left(r\right)$,
\begin{align*}
\delta & =\dim_{H}\left(\mathbf{T}_{1}^{+}\right)\leq\dim_{H}\left(\bigcup_{\ell\geq0}\FI_{\Psi}^{-\ell}\left(\mathbf{T}_{3}^{+}\left(r\right)\right)\right)\\
 & =\sup_{\ell\geq0}\dim_{H}\left(\FI_{\Psi}^{-\ell}\left(\mathbf{T}_{3}^{+}\left(r\right)\right)\right)=\dim_{H}\left(\mathbf{T}_{3}^{+}\left(r\right)\right)\leq\delta.
\end{align*}
Similarly, we have $\mathbf{E}\left(\alpha'\right)\subset\mathbf{T}_{1}^{-}\subset\mathbf{T}_{2}^{-}$
for all $\alpha'<0$ and hence, by Theorem \ref{thm: Multifractal Decomposition},
giving the lower bound
\[
\dim_{H}\left(\mathbf{E}\left(0\right)\right)=\delta_{0}=\sup_{\alpha'<0}\delta_{\alpha'}=\sup_{\alpha'<0}\dim_{H}\left(\mathbf{E}\left(\alpha'\right)\right)\leq\dim_{H}\left(\mathbf{T}_{1}^{-}\right)\leq\dim_{H}\left(\mathbf{T}_{2}^{-}\right).
\]
And for $\mathbf{T}_{3}^{-}\left(r\right),$ $r\in\R$ similarly as
above 
\[
\dim_{H}\left(\mathbf{E}\left(0\right)\right)\leq\dim_{H}\left(\mathbf{T}_{1}^{-}\right)\leq\sup_{\ell\geq0}\dim_{H}\left(\FI_{\Psi}^{-\ell}\left(\mathbf{T}_{3}^{-}\left(r\right)\right)\right)=\dim_{H}\left(\mathbf{T}_{3}^{-}\left(r\right)\right).
\]
For the upper bound we assume without loss of generality that $\delta_{0}<\delta$
(otherwise nothing is to be shown) and fix $0<\varepsilon<\delta-\delta_{0}$.
By Proposition \ref{prop:Properties of delta_alpha}, we have $q_{0}\le0$.
Hence, we obtain the following bound for $n\in\Z$, 
\begin{align*}
\sum_{\substack{\left|\omega\right|>N\\
S_{\omega}\psi<-n
}
}\left|\pi\left[\omega\right]\right|^{\delta_{0}+\epsilon}\le & \phantom{{\e^{-q_{0}n}}}\sum_{\substack{\left|\omega\right|>N\\
S_{\omega}\psi<-n
}
}\e^{(\delta_{0}+\epsilon)S_{\omega}\varphi+q_{0}S_{\omega}\psi}\e^{-q_{0}S_{\omega}\psi}\\
\le & \e^{-q_{0}n}\sum_{\substack{\left|\omega\right|>N\\
S_{\omega}\psi<-n
}
}\e^{(\delta_{0}+\epsilon)S_{\omega}\varphi+q_{0}S_{\omega}\psi_{0}}\\
\leq & \e^{-q_{0}n}\sum_{\substack{\left|\omega\right|>N}
}\e^{(\delta_{0}+\epsilon)S_{\omega}\varphi+q_{0}S_{\omega}\psi_{0}}<\infty,
\end{align*}
which shows $\dim_{H}\left(\mathbf{T}_{3}^{-}\left(r\right)\cap\left[n-r-1,n-r\right]\right)\leq\delta_{0}$,
for all $n\in\Z$. As a consequence of the countable stability of
the Hausdorff dimension we obtain $\dim_{H}\left(\mathbf{T}_{3}^{-}\left(r\right)\right)\leq\delta_{0}$.
In particular, for $r\in\Z$ we use (\ref{eq:T2UnionT3}) and the
countable stability of the Hausdorff dimension once more to finish
the proof.
\end{proof}
\begin{rem}
The last upper bound in the above proof could also been seen via the
general multifractal formalism for limiting behaviour of $\left(S_{n}\psi/S_{n}\varphi\right)$
provided e.g. in \cite{MR2719683,MR2672614} by observing the inclusion,
for $r\in\R$ and $n\in\Z$, 
\[
\mathbf{T}_{3}^{-}\left(r\right)\cap\left[n,n+1\right]\subset\pi\left\{ \omega\in\Sigma\colon\liminf_{\ell\rightarrow\infty}\frac{S_{\ell}\psi}{S_{\ell}\varphi}\leq0\right\} +n.
\]
Let us consider the implicitly defined function $s:\R\to\R$ by $\mathfrak{P}\left(s(q)\varphi+q\psi\right)=0$
as in the proof of Theorem \ref{thm:critical exponent via base pressures}.
By the general multifractal formalism, for $\alpha=\mu_{\delta\varphi}\left(\psi\right)\ge0$,
we have 
\[
\dim_{H}\left(\pi\left\{ \omega\in\Sigma\colon\liminf_{\ell\rightarrow\infty}\frac{S_{\ell}\psi}{S_{\ell}\varphi}\le0\right\} \right)=\inf_{q\in\R}s\left(q\right).
\]
This is to say that we have to find $s_{0}$ and $q_{0}$ with $\mathfrak{P}\left(s_{0}\varphi+q_{0}\psi\right)=0$
and $s'\left(q_{0}\right)=\frac{\partial}{\partial q}\mathfrak{P}\left(s_{0}\varphi+q\psi\right)|_{q=q_{0}}=0$.
By Theorem \ref{thm:critical exponent via base pressures}, this shows
$\inf_{q\in\R}s\left(q\right)=\delta_{0}$. The upper bound then follows
from the countable stability of the Hausdorff dimension.
\end{rem}

\section{Examples and an application to Kleinian groups\label{sec:Examples}}

Let us consider an interval map $F$ with two expansive branches with
slopes $1/c_{1}$ and $1/c_{2}$, respectively, where $c_{1},c_{2}\in(0,1)$
with $c_{1}+c_{2}\leq1$ (see for instance Example \ref{exa:Classical Random Walk}
in the introduction). Then the corresponding geometric potential is
given by $\varphi\left(\omega\right)\coloneqq\log\left(c_{\omega_{1}}\right)$
for $\omega=\left(\omega_{1},\omega_{2},\ldots\right)$. Moreover,
note that $\delta$ solves the Moran-Hutchinson formula $c_{1}^{\delta}+c_{2}^{\delta}=1.$
We determine the dimension spectrum $\alpha\mapsto\delta_{\alpha}$
of the escaping sets for $\FI_{\Psi}$ for different parameters $c_{1},c_{2}$
and step length functions $\Psi$. In the following we also make the
convention that $0\cdot\log0=0$.
\begin{example}
First, we consider arbitrary $c_{1},c_{2}\in(0,1)$ and pick a symmetric
step length function $\Psi$ such that $\psi\left(\omega\right)=\left(-1\right)^{\omega_{1}}$.
Then we have to solve the two equations for $\alpha\in(-1,1)$: 
\begin{align*}
1 & =\exp\left(\mathfrak{P}\left(s\varphi+q\psi_{\alpha}\right)\right)=\e^{-q\alpha}\left(c_{1}^{s}\e^{-q}+c_{2}^{s}\e{}^{q}\right)\eqqcolon z_{\alpha}(s,q)
\end{align*}
and 
\[
0=\frac{\partial z_{\alpha}}{\partial q}(s,q)=-\alpha+\e^{-q\alpha}\left(-c_{1}^{s}\e^{-q}+c_{2}^{s}\e{}^{q}\right).
\]
Using Theorem \ref{thm:critical exponent via base pressures} (i)\&(ii),
this gives for $\alpha\in[-1,1]$ 
\[
\delta_{\alpha}=\delta_{\alpha}(c_{1},c_{2})=-\frac{\left(\frac{1+\alpha}{2}\right)\log\left(\frac{1+\alpha}{2}\right)+\left(\frac{1-\alpha}{2}\right)\log\left(\frac{1-\alpha}{2}\right)}{\left(\frac{1+\alpha}{2}\right)\log(1/c_{1})+\left(\frac{1-\alpha}{2}\right)\log(1/c_{2})}.
\]
Note that this expression is in fact the quotient of the measure-theoretic
entropy over the Lyapunov exponent of $F$ with respect to the $\left(\frac{1+\alpha}{2},\frac{1-\alpha}{2}\right)$-Bernoulli
measure. Moreover, 
\[
\lim_{\alpha\searrow-1}\frac{d}{d\alpha}\delta_{\alpha}=\infty\quad\text{and}\quad\lim_{\alpha\nearrow1}\frac{d}{d\alpha}\delta_{\alpha}=-\infty,
\]
see Figure \ref{fig:asymmetric-Random-Walk-Spectrum} for the graph
of $\alpha\mapsto\delta_{\alpha}$. In order to determine the Hausdorff
dimension of the (uniformly) recurrent set we need to determine $\delta_{0}$
depending on the parameters $\left(c_{1},c_{2}\right)$. We obtain
\[
\delta_{0}\left(c_{1},c_{2}\right)=\frac{\log4}{\log(1/c_{1})+\log(1/c_{2})},
\]
see Figure \ref{fig:delta_0Graph} for a one-parameter plot of $c\mapsto\delta_{0}\left(c,1-c\right)$.

Finally, for $\alpha=0$, using Theorem \ref{thm:fibre-induced pressure via base pressure}
we obtain for the fibre-induced pressure 
\[
\mathcal{P}\left(t\varphi,\psi\right)=\log2+t\cdot\left(\log c_{1}+\log c_{2}\right)/2.
\]
\end{example}

\begin{example}
\label{exa:asymmetric setp}Here, we set $c_{1}=c_{2}=c$ and consider
$\Psi$ such that $\psi\left(\omega\right)=m_{1}(2-\omega_{1})+m_{2}(\omega_{1}-1)$
with $m_{1}<m_{2}$ (a-)symmetric. A similar calculation as in the
first example leads us to the dimension spectrum
\[
\delta_{\alpha}=\delta_{\alpha}(c,m_{1},m_{2})=-\frac{\left(\frac{\alpha-m_{1}}{m_{2}-m_{1}}\right)\log\left(\frac{\alpha-m_{1}}{m_{2}-m_{1}}\right)+\left(\frac{m_{2}-\alpha}{m_{2}-m_{1}}\right)\log\left(\frac{m_{2}-\alpha}{m_{2}-m_{1}}\right)}{\log(1/c)},
\]
for $\alpha\in[m_{1},m_{2}]$. Again, note that this expression is
the quotient of the measure-theoretic entropy over the Lyapunov exponent
of $F$ with respect to the $\left(\frac{\alpha-m_{1}}{m_{2}-m_{1}},\frac{m_{2}-\alpha}{m_{2}-m_{1}}\right)$-Bernoulli
measure.
\end{example}

\begin{example}
\label{exa:more intervals}For the last example we change the map
$F$. Assume that $F$ has $g_{1}+g_{2}$ intervals where $g_{1}$,
$g_{2}\geq1$ such that on each interval the slope $1/c\geq g_{1}+g_{2}$.
Further, we choose $\Psi$ in such a way that on $g_{1}$ intervals
it is $-1$ and on the others it is $+1$. An analogous calculation
as in the previous examples gives for $\alpha\in[-1,1]$ the following
dimension spectrum
\[
\delta_{\alpha}=\delta_{\alpha}(c,g_{1},g_{2})=-\frac{g_{1}\cdot\left(\frac{1-\alpha}{2g_{1}}\right)\log\left(\frac{1-\alpha}{2g_{1}}\right)+g_{2}\cdot\left(\frac{1+\alpha}{2g_{2}}\right)\log\left(\frac{1+\alpha}{2g_{2}}\right)}{\log(1/c)}.
\]
Once more, observe that this expression is a quotient of the measure-theoretic
entropy over the Lyapunov exponent of $F$ for a corresponding Bernoulli
measure. For a particular choice consider $g_{1}=1$, $g_{2}=2$ and
$c=1/3$, see also Figure \ref{fig:more examples}. We have $\dim_{H}\left(\mathbf{E}\left(1\right)\right)=\log2/\log3$
and this means that the Hausdorff dimension of $\mathbf{E}\left(1\right)$
coincides with the dimension of the $1/3$-Cantor set. 
\begin{figure}[h]
\pgfplotsset{width=\textwidth*0.48,compat=1.9}
\begin{tikzpicture}[line cap=round,line join=round,>=triangle 45]
\begin{axis}[xmin=-1.2,xmax=2.2,ymin=-0.2,ymax=1.1, 
 axis lines=middle,     axis line style={->},     tick style={color=black,line width=1.1pt},     xtick={-1,0,1,2}, xticklabels={$-1$,$0$,$1$,$2$},  ytick={0,1},yticklabels={$0$,$1$} ]     
\addplot 
[         domain=-1:2, samples=270,line width=1.2pt,  
] {(ln(3)/ln(2)-((x+1)*ln(x+1)+(2-x)*ln(2-x))/(3*ln(2)))}; 
\addplot 
[         domain=0:1, samples=197,line width=1.2pt, dashed 
] {(-(x*ln(x)+(1-x)*ln(1-x))/(ln(2)))}; 
\end{axis}
\end{tikzpicture}
\;\;
\begin{tikzpicture}[line cap=round,line join=round,>=triangle 45]
\begin{axis}[xmin=-1.2,xmax=1.2, ymin=-0.2,ymax=1.1,    
  axis lines=middle,     axis line style={->},     tick style={color=black,line width=1.1pt},     xtick={-1,0,1}, xticklabels={$-1$,$0$,$1$},  ytick={0,1},yticklabels={$0$, $1$} ]     
\addplot 
[         domain=-1:1, samples=197,line width=1.2pt,  
] {(-((1-x)*ln((1-x)/2)+(1+x)*ln((1+x)/4))/(2*ln(3)))}; 
\addplot 
[         domain=-1:1, samples=197,line width=1.2pt, dashed 
] {(-((1-x)*ln((1-x)/16)+(1+x)*ln((1+x)/4))/(2*ln(10)))}; 
\end{axis}
\end{tikzpicture}

\caption{The escape rate spectrum $\alpha\protect\mapsto\delta_{\alpha}$ for
Example \ref{exa:asymmetric setp} on the left (solid line with $c=1/2$,
$m_{1}=-1$, $m_{2}=2$ and dashed line with $c=1/2$, $m_{1}=0$,
$m_{2}=1$) and for Example \ref{exa:more intervals} on the right
(solid line with $c=1/3$, $g_{1}=1$, $g_{2}=2$ and dashed line
with $c=1/10$, $g_{1}=8$, $g_{2}=2$). \label{fig:more examples}}
\end{figure}
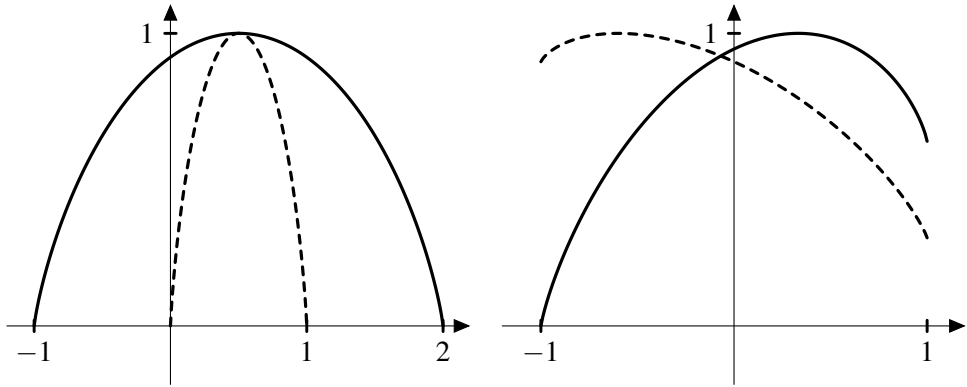

Finally, we provide an application of our results partially recovering
a result of \cite{rees_1981} about extensions of Kleinian groups.
Recall that if $G=\left\langle g_{1},\dots,g_{k}\right\rangle $ is
a Schottky group, the limit set of $G$ can be represented by the
subshift of finite type $\Sigma_{G}:=\left\{ \omega=(\omega_{1},\omega_{2}\dots)\in I^{\N}\mid\omega_{i}\neq-\omega_{i+1},\,\,i\in\N\right\} $
with alphabet $I=\left\{ \pm1,\dots\pm k\right\} $. Then $\pi$ denotes
the natural coding map of the limit set with respect to $\Sigma_{G}$.
Further, the $\sigma$-invariant $\delta_{G}$-dimensional Patterson-Sullivan
measure of $G$ is given by the $\sigma$-invariant Gibbs measure
on $\Sigma_{G}$ with respect to the Hölder continuous geometric potential
$\delta_{G}\cdot\varphi:\Sigma_{G}\rightarrow\R$ associated with
$G$ (\cite{Bowen1979}). Concerning the general theory of the Patterson-Sullivan
measure and its $\sigma$-invariant version, see \cite{MR0450547}
and \cite{MR556586}, respectively.
\end{example}

\begin{prop}
Let $G=\left\langle g_{1},\dots,g_{k}\right\rangle $ be a Schottky
group. Let $\mu$ be the $\sigma$-invariant version of the $\delta_{G}$-dimensional
Patterson-Sullivan measure on the subshift of finite type $\Sigma_{G}$
with alphabet $I=\left\{ \pm1,\dots\pm k\right\} $. Then we have
$\mu\left(\left[i\right]\right)=\mu\left(\left[-i\right]\right)$
for all $i\in I$.
\end{prop}

\begin{proof}
Denote by $m$ the $\delta_{G}$-dimensional Patterson-Sullivan measure.
Recall that $|\xi-\eta|^{-2\delta_{G}}dm(\xi)\times dm(\eta)$ defines
a $G$-invariant measure on the geodesics on $\mathbb{H}^{n}\backslash G$
represented by $\Lambda(G)\times\Lambda(G)$. By disintegration of
this measure, we obtain the $\sigma$-invariant version $\mu$ on
$\Sigma_{G}$ which satisfies for all $i\in I$, 
\[
\mu([i])=\int_{\pi([i])}\sum_{j\neq i}\int_{\pi([j])}\left|\xi-\eta\right|^{-2\delta_{G}}dm(\eta)dm(\xi).
\]
Since 
\[
g_{i}\left\{ \pi[i]\times\left(\bigcup_{j\neq i}\pi[j]\right)\right\} =\left(\bigcup_{j\neq-i}\pi[j]\right)\times\pi[-i],
\]
the $G$-invariance of $|\xi-\eta|^{-2\delta_{G}}dm(\xi)\times dm(\eta)$
implies 
\begin{align*}
\mu([i]) & =\int_{\pi([i])}\sum_{j\neq i}\int_{\pi([j])}\left|\xi-\eta\right|^{-2\delta_{G}}dm(\eta)dm(\xi)\\
 & =\int_{\pi([-i])}\sum_{j\neq-i}\int_{\pi([j])}\left|\xi-\eta\right|^{-2\delta_{G}}dm(\eta)dm(\xi)=\mu([-i]).
\end{align*}
This shows the assertion.
\end{proof}
Now, let $G$ be a Schottky group and $N<G$ a normal subgroup such
that $G/N=\left\langle g\right\rangle \cong\Z$ and without loss of
generality $g\in\left\{ g_{1},\ldots,g_{k}\right\} $. Then, using
Kronecker delta notation, we set $\psi:\Sigma_{G}\rightarrow\Z$,
$\psi(\omega):=\delta_{g_{\omega_{1}},g}-\delta_{g_{\omega_{1}},g^{-1}}$.
Relating the hyperbolic distance to $\varphi$ as in \cite{MR2041265}
as well as replacing the fullshift $\Sigma$ with $\Sigma_{G}$ in
the definition of Poincaré exponents, we deduce that $\delta=\delta_{G}$
and $\delta_{0}=\delta_{N}$. Further, note that $0\in\left(\underline{\psi},\overline{\psi}\right)=(-1,1)$
and that the first and the last assertion of Theorem \ref{thm:fibre-induced pressure via base pressure}
and Corollary \ref{cor:full pressure if no drift} remain valid also
for mixing subshifts of finite type. Hence, combining the previous
proposition with Corollary \ref{cor:full pressure if no drift} and
Theorem \ref{thm:fibre-induced pressure via base pressure} applied
to $\delta_{G}\varphi$, gives Corollary \ref{thm:Kleingroup divergence type}.


\begin{thebibliography}{10}

\bibitem{MR0419727}
G.~Atkinson.
\newblock Recurrence of co-cycles and random walks.
\newblock {\em J. London Math. Soc. (2)}, 13(3):486--488, 1976.

\bibitem{MR2373353}
A.~Avila and M.~Lyubich.
\newblock Hausdorff dimension and conformal measures of {F}eigenbaum {J}ulia
  sets.
\newblock {\em J. Amer. Math. Soc.}, 21(2):305--363, 2008.

\bibitem{MR1837214}
L.~Barreira and B.~Saussol.
\newblock Variational principles and mixed multifractal spectra.
\newblock {\em Trans. Amer. Math. Soc.}, 353(10):3919--3944, 2001.

\bibitem{MR1484767}
C.~J. Bishop and P.~W. Jones.
\newblock Hausdorff dimension and {K}leinian groups.
\newblock {\em Acta Math.}, 179(1):1--39, 1997.

\bibitem{Bowen1979}
R.~Bowen.
\newblock Hausdorff dimension of quasi-circles.
\newblock {\em Publications Math{\'e}matiques de l'IH{\'E}S}, 50:11--25, 1979.

\bibitem{MR783536}
R.~Brooks.
\newblock The bottom of the spectrum of a {R}iemannian covering.
\newblock {\em J. Reine Angew. Math.}, 357:101--114, 1985.

\bibitem{MR2959300}
H.~Bruin and M.~Todd.
\newblock Transience and thermodynamic formalism for infinitely branched
  interval maps.
\newblock {\em J. Lond. Math. Soc. (2)}, 86(1):171--194, 2012.

\bibitem{MR2861747}
V.~Cyr.
\newblock Countable {M}arkov shifts with transient potentials.
\newblock {\em Proc. Lond. Math. Soc. (3)}, 103(6):923--949, 2011.

\bibitem{MR2551790}
V.~Cyr and O.~Sarig.
\newblock Spectral gap and transience for {R}uelle operators on countable
  {M}arkov shifts.
\newblock {\em Comm. Math. Phys.}, 292(3):637--666, 2009.

\bibitem{MR2097162}
K.~Falk and B.~O. Stratmann.
\newblock Remarks on {H}ausdorff dimensions for transient limit sets of
  {K}leinian groups.
\newblock {\em Tohoku Math. J. (2)}, 56(4):571--582, 2004.

\bibitem{MR3286501}
A.-H. Fan, T.~Jordan, L.~Liao, and M.~Rams.
\newblock Multifractal analysis for expanding interval maps with infinitely
  many branches.
\newblock {\em Trans. Amer. Math. Soc.}, 367(3):1847--1870, 2015.

\bibitem{MR1916371}
D.-J. Feng, K.-S. Lau, and J.~Wu.
\newblock Ergodic limits on the conformal repellers.
\newblock {\em Adv. Math.}, 169(1):58--91, 2002.

\bibitem{JKG20}
M.~Gr{\"o}ger, J.~Jaerisch, and M.~Kesseb{{\"o}}hmer.
\newblock Escaping sets for transient dynamics with reflective boundary.
\newblock {\em preprint}, 2019.

\bibitem{MR3610938}
G.~Iommi, T.~Jordan, and M.~Todd.
\newblock Transience and multifractal analysis.
\newblock {\em Ann. Inst. H. Poincar{\'e} Anal. Non Lin{\'e}aire},
  34(2):407--421, 2017.

\bibitem{MR3299281}
J.~Jaerisch.
\newblock A lower bound for the exponent of convergence of normal subgroups of
  {K}leinian groups.
\newblock {\em J. Geom. Anal.}, 25(1):298--305, 2015.

\bibitem{MR3436756}
J.~Jaerisch.
\newblock Recurrence and pressure for group extensions.
\newblock {\em Ergod. Theory Dyn. Syst.}, 36(1):108--126, 2016.

\bibitem{MR2672614}
J.~Jaerisch and M.~Kesseb{\"o}hmer.
\newblock The arithmetic-geometric scaling spectrum for continued fractions.
\newblock {\em Ark. Mat.}, 48(2):335--360, 2010.

\bibitem{MR2719683}
J.~Jaerisch and M.~Kesseb{\"o}hmer.
\newblock Regularity of multifractal spectra of conformal iterated function
  systems.
\newblock {\em Trans. Amer. Math. Soc.}, 363(1):313--330, 2011.

\bibitem{MR3190211}
J.~Jaerisch, M.~Kesseb{\"o}hmer, and S.~Lamei.
\newblock Induced topological pressure for countable state {M}arkov shifts.
\newblock {\em Stoch. Dyn.}, 14(2):1350016, 31, 2014.

\bibitem{JordanRams2017}
T.~Jordan and M.~Rams.
\newblock Birkhoff spectrum for piecewise monotone interval maps.
\newblock {\em arXiv:1712.03750}, pages 1--22, 2017.

\bibitem{MR2197375}
B.~Karpi\'{n}ska and M.~Urba\'{n}ski.
\newblock How points escape to infinity under exponential maps.
\newblock {\em J. London Math. Soc. (2)}, 73(1):141--156, 2006.

\bibitem{MR2041265}
M.~Kesseb\"{o}hmer and B.~O. Stratmann.
\newblock A multifractal formalism for growth rates and applications to
  geometrically finite {K}leinian groups.
\newblock {\em Ergodic Theory Dynam. Systems}, 24(1):141--170, 2004.

\bibitem{MR0112053}
H.~Kesten.
\newblock Full {B}anach mean values on countable groups.
\newblock {\em Math. Scand.}, 7:146--156, 1959.

\bibitem{MR0109367}
H.~Kesten.
\newblock Symmetric random walks on groups.
\newblock {\em Trans. Amer. Math. Soc.}, 92:336--354, 1959.

\bibitem{MR2003772}
R.~D. Mauldin and M.~Urba{{\'n}}ski.
\newblock {\em Graph directed {M}arkov systems}, volume 148 of {\em Cambridge
  Tracts in Mathematics}.
\newblock Cambridge University Press, Cambridge, 2003.
\newblock Geometry and dynamics of limit sets.

\bibitem{MR871679}
C.~McMullen.
\newblock Area and {H}ausdorff dimension of {J}ulia sets of entire functions.
\newblock {\em Trans. Amer. Math. Soc.}, 300(1):329--342, 1987.

\bibitem{MR0450547}
S.~J. Patterson.
\newblock The limit set of a {F}uchsian group.
\newblock {\em Acta Math.}, 136(3-4):241--273, 1976.

\bibitem{MR1858487}
Y.~Pesin and H.~Weiss.
\newblock The multifractal analysis of {B}irkhoff averages and large
  deviations.
\newblock In {\em Global analysis of dynamical systems}, pages 419--431. Inst.
  Phys., Bristol, 2001.

\bibitem{rees_1981}
M.~Rees.
\newblock {Checking ergodicity of some geodesic flows with infinite Gibbs
  measure}.
\newblock {\em Ergod. Theory Dyn. Syst.}, 1(1):107--133, 1981.

\bibitem{MR2465667}
L.~Rempe.
\newblock Hyperbolic dimension and radial {J}ulia sets of transcendental
  functions.
\newblock {\em Proc. Amer. Math. Soc.}, 137(4):1411--1420, 2009.

\bibitem{MR2166367}
T.~Roblin.
\newblock Un th\'eor\`eme de {F}atou pour les densit\'es conformes avec
  applications aux rev\^etements galoisiens en courbure n\'egative.
\newblock {\em Israel J. Math.}, 147:333--357, 2005.

\bibitem{rockafellar1970}
R.~T. Rockafellar.
\newblock {\em Convex Analysis}.
\newblock Princeton landmarks in mathematics and physics. Princeton University
  Press, 1970.

\bibitem{MR0234697}
D.~Ruelle.
\newblock Statistical mechanics of a one-dimensional lattice gas.
\newblock {\em Comm. Math. Phys.}, 9:267--278, 1968.

\bibitem{MR1738951}
O.~M. Sarig.
\newblock Thermodynamic formalism for countable {M}arkov shifts.
\newblock {\em Ergod. Theory Dyn. Syst.}, 19(6):1565--1593, 1999.

\bibitem{MR1818392}
O.~M. Sarig.
\newblock Thermodynamic formalism for null recurrent potentials.
\newblock {\em Israel J. Math.}, 121:285--311, 2001.

\bibitem{MR0578731}
K.~Schmidt.
\newblock {\em Cocycles on ergodic transformation groups}.
\newblock Macmillan Company of India, Ltd., Delhi, 1977.
\newblock Macmillan Lectures in Mathematics, Vol. 1.

\bibitem{MR1626737}
M.~Shishikura.
\newblock The {H}ausdorff dimension of the boundary of the {M}andelbrot set and
  {J}ulia sets.
\newblock {\em Ann. of Math. (2)}, 147(2):225--267, 1998.

\bibitem{MR3922537}
O.~Shwartz.
\newblock Thermodynamic formalism for transient potential functions.
\newblock {\em Comm. Math. Phys.}, 366(2):737--779, 2019.

\bibitem{Stadlbauer11}
M.~Stadlbauer.
\newblock An extension of {K}esten's criterion for amenability to topological
  {M}arkov chains.
\newblock {\em Adv. Math.}, 235:450--468, 2013.

\bibitem{MR1438267}
B.~Stratmann and R.~Vogt.
\newblock Fractal dimensions for dissipative sets.
\newblock {\em Nonlinearity}, 10(2):565--577, 1997.

\bibitem{MR556586}
D.~Sullivan.
\newblock The density at infinity of a discrete group of hyperbolic motions.
\newblock {\em Inst. Hautes \'Etudes Sci. Publ. Math.}, 50:171--202, 1979.

\bibitem{MR2302520}
M.~Urba\'{n}ski and A.~Zdunik.
\newblock Geometry and ergodic theory of non-hyperbolic exponential maps.
\newblock {\em Trans. Amer. Math. Soc.}, 359(8):3973--3997, 2007.

\bibitem{MR648108}
P.~Walters.
\newblock {\em An introduction to ergodic theory}, volume~79 of {\em Graduate
  Texts in Mathematics}.
\newblock Springer-Verlag, New York, 1982.

\end{thebibliography}
\end{document}